\documentclass[11pt]{amsart}
\usepackage[english]{babel}
\usepackage{amsmath,amsthm,amssymb,xcolor,enumerate}
\usepackage{euscript,mathrsfs}
\usepackage[left=2cm,right=2cm,top=3.5cm,bottom=3.5cm]{geometry}
\usepackage{amscd}
\usepackage[citecolor=blue,colorlinks=true]{hyperref}
\usepackage{framed}
\usepackage[framemethod=tikz]{mdframed}

\allowdisplaybreaks

\numberwithin{equation}{section}

\newtheorem{Theorem}{Theorem}[section]
\newtheorem{Proposition}[Theorem]{Proposition}
\newtheorem{Lemma}[Theorem]{Lemma}
\newtheorem{Corollary}[Theorem]{Corollary}

\theoremstyle{definition}
\newtheorem{Definition}[Theorem]{Definition}

\newtheorem{Remark}[Theorem]{Remark}

\newcommand{\bTheorem}[1]{
\begin{Theorem} \label{T#1} }
\newcommand{\eT}{\end{Theorem}}

\newcommand{\bProposition}[1]{
\begin{Proposition} \label{P#1}}
\newcommand{\eP}{\end{Proposition}}

\newcommand{\bLemma}[1]{
\begin{Lemma} \label{L#1} }
\newcommand{\eL}{\end{Lemma}}

\newcommand{\bCorollary}[1]{
\begin{Corollary} \label{C#1} }
\newcommand{\eC}{\end{Corollary}}

\newcommand{\bRemark}[1]{
\begin{Remark} \label{R#1} }
\newcommand{\eR}{\end{Remark}}

\newcommand{\bDefinition}[1]{
\begin{Definition} \label{D#1} }
\newcommand{\eD}{\end{Definition}}

\newcommand{\bfA}{\mathbf{A}}
\newcommand{\bfQ}{\mathbf{Q}}
\newcommand{\ds}{\,\mathrm{d}s}
\newcommand{\dt}{\,\mathrm{d}t}
\newcommand{\dx}{\,\mathrm{d}x}
\newcommand{\dxt}{\,\mathrm{d}x\,\mathrm{d}t}
\newcommand{\dxs}{\,\mathrm{d}x\,\mathrm{d}s}
\newcommand{\dif}{\mathrm{d}}
\newcommand{\D}{\mathrm{d}}
\newcommand{\dd}{\mathrm{d}}
\newcommand{\Div}{\mathrm{div}}
\newcommand{\vr}{\varrho}

\newcommand{\bfu}{\mathbf{u}}

\newcommand{\bfw}{\mathbf{w}}

\newcommand{\bfW}{\mathbf{W}}
\newcommand{\bfv}{\mathbf{v}}
\newcommand{\bff}{\mathbf{f}}
\newcommand{\bfV}{\mathbf{V}}
\newcommand{\rr}{\mathbf{r}_t}

\newcommand{\bfq}{\mathbf{q}}
\newcommand{\bfphi}{\boldsymbol{\varphi}}
\newcommand{\bfeta}{\boldsymbol{\eta}}
\newcommand{\bfvarphi}{\boldsymbol{\varphi}}

\newcommand{\bfpsi}{\boldsymbol{\psi}}

\newcommand{\TN}{\mathcal{T}^N}
\newcommand{\mt}{\mathcal{T}^N}
\newcommand{\mf}{\mathfrak{F}}

\newcommand{\prst}{\mathbb P}
\newcommand{\p}{\mathbb P}

\newcommand{\FF}{\mathfrak F}
\newcommand{\pas}{$\mathbb P$-a.s.}
\newcommand{\Grad}{\nabla}

\renewcommand{\div}{{\rm div\,}}
\newcommand{\T}{\mathbb{T}}
\newcommand{\R}{\mathbb{R}}
\newcommand{\N}{\mathbb{N}}

\newcommand\Cbox[2]{%
    \newbox\contentbox%
    \newbox\bkgdbox%
    \setbox\contentbox\hbox to \hsize{%
        \vtop{
            \kern\columnsep
            \hbox to \hsize{%
                \kern\columnsep%
                \advance\hsize by -2\columnsep%
                \setlength{\textwidth}{\hsize}%
                \vbox{
                    \parskip=\baselineskip
                    \parindent=0bp
                    #2
                }%
                \kern\columnsep%
            }%
            \kern\columnsep%
        }%
    }%
    \setbox\bkgdbox\vbox{
        \color{#1}
        \hrule width  \wd\contentbox %
               height \ht\contentbox %
               depth  \dp\contentbox
        \color{black}
    }%
    \wd\bkgdbox=0bp%
    \vbox{\hbox to \hsize{\box\bkgdbox\box\contentbox}}%
    \vskip\baselineskip%
}

\mdfdefinestyle{MyFrame}{%
	linecolor=black,
	outerlinewidth=1pt,
	roundcorner=5pt,
	innertopmargin=\baselineskip,
	innerbottommargin=\baselineskip,
	innerrightmargin=10pt,
	innerleftmargin=10pt,
	backgroundcolor=white!20!white}

\date{}

\begin{document}

\title{Compressible fluids excited by space-dependent transport noise}

\author{Dominic Breit}
\address[D. Breit]{
Institute of Mathematics, TU-Clausthal, Erzstra\ss e1, 38678 Clausthal-Zellerfeld, Germany}
\email{dominic.breit@tu-clausthal.de}

\author{Eduard Feireisl}
\address[E. Feireisl]{Institute of Mathematics AS CR, \v{Z}itn\'a 25, 115 67 Praha 1, Czech Republic}
\email{feireisl@math.cas.cz}

\author{Martina Hofmanov\'a}
\address[M. Hofmanov\'a]{Fakult\"at f\"ur Mathematik, Universit\"at Bielefeld, D-33501 Bielefeld, Germany}
\email{hofmanova@math.uni-bielefeld.de}

\author{Piotr B. Mucha}
\address[P.B. Mucha]{Institute of Applied Mathematics and Mechanics, University of Warsaw, ul. Banach 2,
02-097 Warsaw, Poland}
\email{p.mucha@mimuw.edu.pl}

\begin{abstract}
We study the compressible Navier--Stokes system driven by physically relevant transport noise, where the noise influences both the continuity and momentum equations. 
Our approach is based on transforming the system into a partial differential equation with random, time- and space-dependent coefficients. A key challenge arises from the fact that these coefficients are non-differentiable in time, rendering standard compactness arguments for the identification of the pressure inapplicable. To overcome this difficulty, we develop a novel multi-layer approximation scheme and introduce a precise localization strategy with respect to both the sample space and time variable. The limit pressure is then identified via the corresponding effective viscous flux identity.
By means of  stochastic compactness methods, particularly Skorokhod's representation theorem and its generalization by Jakubowski, we ensure the progressive measurability required to return to the original system. 
Our results broaden the applicability of transport noise models in fluid dynamics and offer new insights into the interaction between stochastic effects and compressibility.
%
%
%
\end{abstract}


\maketitle

\date{\today}

\maketitle

\section{Introduction}

We study the stochastic compressible Navier--Stokes equations, which describe the motion of a compressible viscous fluid. 
To avoid technical problems related to the boundary conditions, we suppose the motion is space periodic 
meaning confined to	an $N$-dimensional flat torus $\TN$ in dimensions $N=2,3$. The system consists of the continuity equation for the fluid density $\varrho:\Omega\times[0,T]\times\TN\to R$ and the momentum equation for the  velocity field $\bfu:\Omega\times[0,T]\times\TN\to R^{N}$ given by
\begin{align}
\partial_t\varrho+\Div (\varrho\bfu)&=\dot{\xi},\label{1.1}\\
\partial_t (\varrho\bfu)+\Div\,\big(\varrho\bfu\otimes\bfu\big)+\nabla p&=\Div\,\mathbb S(\nabla\bfu)+\dot{\bfeta},\label{1.1b}
\end{align}
 supplemented with the initial conditions $(\varrho(0),\varrho\bfu(0))=(\varrho_{0},\bfq_{0})$.
Here $\dot\xi$ and $\dot{\bfeta}$  represent  noise terms that introduce randomness into the system modelled by the variable $\omega\in\Omega$. These noise terms are typically formulated  as time derivatives of a certain Brownian motion on a probability space $(\Omega,\mf,\prst)$, which leads to solutions that are almost surely non-differentiable in time. Consequently, the interpretation of the equations differs from their deterministic counterparts. Additionally, the noise terms  may depend on the solution  itself. In view of the physical interpretation of the equation of continuity \eqref{1.1} as mass conservation, the term $\dot{\xi}$ must 
be in the form of a spatial divergence, such that it complies with the constraint $\int_{\TN}\dot{\xi}\dx=0$.

The stress tensor follows Newton's rheological law
\begin{align}\label{eq:nr}
\mathbb S(\nabla\bfu)=\mu\big(\nabla\bfu+\nabla\bfu^\top \big)+\lambda\Div\bfu\,\mathbb I
\end{align}
with strictly positive viscosity coefficients $\mu$ and $\lambda$. 

The pressure obeys the isentropic pressure law $p=p(\varrho)=a\varrho^\gamma$, where $a > 0$ and $\gamma>1$ stands for the adiabatic exponent.

The deterministic version of these equations, corresponding to  $\dot{\xi}=0$ and $\dot{\bfeta}=0$ or $\dot{\bfeta}= \varrho \bff$ with a given deterministic force $\bff$, has been extensively studied.
The existence of weak solutions satisfying the corresponding energy inequality was first shown by Lions \cite{Li2} under the assumption $\gamma\geq 3N/(N+2)$ and later improved by Feireisl et al. \cite{feireisl1} to the range  $\gamma>N/2$. The borderline 
case  $\gamma = 1$ in two dimensions was solved recently in \cite{PlWe}. Although physically relevant, it seems currently out of reach to further relax this condition. 
The uniqueness of weak solutions remains one of the central open problems in the field. While convex integration techniques have led to non-uniqueness results for weak solutions \cite{BV} in the incompressible setting (though not Leray--Hopf solutions, i.e. not necessarily satisfying the energy inequality), the compressible case has proven significantly more difficult, resisting all attempts so far. Due to recent results of Merle et al. \cite{MRRSz} we know that smooth solutions to the compressible Navier--Stokes system can develop a singular profile.

\bigskip

Introducing stochastic noise into the equations leads to additional mathematical challenges and enriches the range of possible research questions. Depending on its specific form, the noise can have different physical interpretations, which, in turn, dictate the appropriate mathematical formulation.

\subsubsection*{Stochastic forcing} 

One  common way to introduce a stochastic perturbation is in the form of a stochastic forcing given by 
 \begin{align*}
\dot{\xi}=0,\qquad \dot{\bfeta}=\Phi(\varrho,\varrho\bfu)\,\frac{\dd W}{\dt},
\end{align*}
with a (possibly infinite-dimensional) Brownian motion $W$ and an operator $\Phi$ satisfying appropriate growth assumptions. The simplest case is additive noise, where $\Phi(\varrho,\varrho\bfu)=\varrho\Psi$ with $\Psi$ independent of $\varrho$ and $\bfu$. The existence of weak solutions to \eqref{1.1}--\eqref{1.1b} in this case was shown in \cite{feireisl2}, relying  on a semi-deterministic approach. A fully stochastic framework was later developed in
 \cite{BrHo} via the concept of weak martingale solutions, where the probability space itself is part of the solution. This builds the cornerstone of the systematic study in the monograph \cite{BrFeHobook}.

\subsubsection*{Transport noise} 

An alternative form of noise, referred to as transport noise, is given by
\begin{align}\label{eq:transport}
\dot{\xi}=\sum_{k=1}^K\Div(\varrho\bfQ_k) \circ\frac{\dd W_k}{\dt}, \qquad\dot{\bfeta}=\sum_{k=1}^K\Div({\varrho\bfu\otimes\bfQ_k})\circ\frac{\dd W_k}{\dt},
\end{align}
with  given smooth vector fields $\bfQ_k = \bfQ_k(x), x \in \TN$, $k=1,\dots, K$, and the interpretation
$$
[\Div(\varrho\bfu\otimes\bfQ_k)]_{i}=\partial_{x_{j}}(\varrho\bfu^{i}\bfQ_{k,j}),\quad \Div(\varrho\bfQ_k)=\partial_{x_{j}}(\varrho\bfQ_{k,j}).
$$
The stochastic differential is driven by a $K$-dimensional Wiener process $W=(W_{1},\dots,W_{K})$ and is understood in  Stratonovich's sense. Note that under appropriate summability condition, an infinite sum is possible as well.  If the vector fields $\bfQ_k$, $k=1,\dots, K$, are solenoidal,
the noise is energy conservative, ensuring that the energy inequality remains valid pathwise, without the need for taking the expectation as it is the case in the setting of stochastic forcing.

It is worth noting that stochasticity in the  model \eqref{eq:transport} is not an external force but  an intrinsic property of the
system. Unlike systems  with stochastic forcing, whose interpretation as turbulence models may be questionable (see \cite{FeiHof22}), the transport noise in the Navier--Stokes system has a solid physical foundation. The  motivation for considering transport noise is twofold:
\begin{itemize}
\item Holm derived   fluid dynamics models with transport noise from  fundamental  physical principles to account for turbulent effects  (see \cite{HOLM2,HOLM,HOLM1}, particularly \cite{HOLM2} for the compressible case). The determination of  the vector fields $\bfQ_{k}$, $k=1,\dots,K$ is discussed in \cite{Co1,Co2}.
\item Transport noise exhibits  regularisation effects, as shown in \cite{FGP} for the transport equation and more recently in \cite{FL} for the three-dimensional incompressible Navier--Stokes equations.
\end{itemize}

The only known existence result for solutions to \eqref{1.1}--\eqref{eq:transport} is our recent work \cite{BFHZ}, which addresses the special case of constant vector fields $\bfQ_k\in {R}^N$, $k=1,\dots,K$. This restriction significantly simplifies a crucial step in identifying the limit solution. Specifically, the compensated compactness argument leading to Lions' identity,
\begin{align}\label{eq:fluxpsiintro}
\overline{p(\varrho) \varrho } - \overline{p(\varrho)} \varrho = 
(\lambda + 2 \mu) \left( \overline{\varrho {\rm div} {\bf u} } - \varrho  {\rm div} {\bf u} \right),
\end{align}
where bar denotes the weak limits of composed functions, 
remains unaffected by the stochastic terms.
Since the pioneering work \cite{Li2}, this identity has been a basic tool for  proving compactness of the density of approximate solutions. Consequently, if $\bfQ_k\in {R}^N$, $k=1,\dots,K$, are constant, the deterministic approach remains applicable.

In the general case of physically relevant space-dependent vector fields $\bfQ_k$, $k=1,\dots,K$, the noise directly influences the computations.
While this poses no difficulty for It\^o's noise of the stochastic forcing type, as it vanishes in expectation, the situation is markedly different for transport noise \eqref{eq:transport}. Here, the cancellation property no longer holds, making compactness arguments significantly more challenging and far from straightforward.

\smallskip

In this paper, we introduce several new ideas to address this challenge. Our primary approach is a Lagrangian reformulation of \eqref{1.1}--\eqref{eq:transport} using the so-called flow transformation. This method has already proven effective for various problems in stochastic partial differential equations. In the context of fluid dynamics, we particularly highlight \cite{HLP}, where the formalism employed in this work was first introduced.
The core idea is to define a change of coordinates via the stochastic differential equation
\begin{align}\label{eq:flowsdeintro}
\bfphi(t,x) = x - \sum_{k=1}^K \int_0^t \bfQ_k(\bfphi(x,s)) \circ \D W_k.
\end{align}
This equation admits a unique probabilistically strong solution, generating a flow of stochastic diffeomorphisms. If the vector fields $\bfQ_k$, $k=1,\dots,K$, are solenoidal, these diffeomorphisms are measure preserving. Moreover, if the vector fields are smooth, the diffeomorphisms belong to $C^{\infty}$.

Setting $\bfv(t,x):=\bfu(t,\bfphi(t,x))$ and $\eta(t,x):=\varrho(t,\bfphi(t,x))$, the system \eqref{1.1}--\eqref{1.1b} transforms into a system of  partial differential equations with random, time- and space-dependent coefficients
\begin{align}
\partial_t \eta + \Div^{\bfphi} (\eta\bfv)  &=0,\label{1.2a}\\
\partial_t (\eta\bfv) + \Div^{\bfphi}\big(\eta\bfv\otimes\bfv\big)  + \nabla^{\bfphi} p(\eta)  &=\Div^{\bfphi}\mathbb S(\nabla^{\bfphi}\bfv).\label{1.2b}
\end{align}
Here the new differential operators with superscript $\bfphi$ are defined as
\begin{align*}
\Div^{\bfphi} \bfpsi=[\Div(\bfpsi\circ\bfphi^{-1})]\circ\bfphi,\quad \nabla^{\bfphi} \psi=[\nabla(\psi\circ\bfphi^{-1})]\circ\bfphi ,
\end{align*}
with further details provided in  Section~\ref{s:2.3}.

We establish  a one-to-one correspondence between weak solutions of the system \eqref{1.1}--\eqref{1.1b} with \eqref{eq:transport} and weak solutions to the transformed system  \eqref{1.2a}--\eqref{1.2b}, provided the solutions are progressively measurable with respect to a filtration $(\mf_{t})_{t\geq0}$ for which  $W=(W_{1},\dots,W_{K})$ is an $(\mf_{t})_{t\geq0}$-Brownian motion. Notably,  in our construction, this filtration is not the canonical filtration generated by the Brownian motion alone, but rather a joint filtration generated by both the Brownian motion and the solution --  an important consideration in the absence of uniqueness. Progressive measurability is essential for the  definition of the stochastic integral in \eqref{1.1}--\eqref{1.1b} with \eqref{eq:transport}. Consequently, solving  the transformed system \eqref{1.2a}--\eqref{1.2b} using purely  deterministic arguments, as  is often possible in other applications of the flow transformation, is insufficient. Instead, we  employ stochastic compactness techniques, leveraging  Skorokhod's representation theorem, or more precisely,   its generalization by Jakubowski \cite[Section~2.8]{BrFeHobook}. This approach ensures the progressive  measurability of solutions, ultimately allowing us to return to  the original system \eqref{1.1}--\eqref{1.1b}.

To solve the transformed system \eqref{1.2a}--\eqref{1.2b}, we develop a novel multi-layer approximation scheme that incorporates several regularizations. Among these are carefully chosen differential operators that remain untransformed by the flow $\bfphi$. These additional terms are systematically eliminated in successive limit passages. As noted earlier, our approach relies on stochastic compactness. 
A key challenge arises from the limited time regularity of the coefficients in \eqref{1.2a}--\eqref{1.2b}. In particular, the system contains coefficients that depend on $\bfphi$, and thus on both $\bfQ_k$ and $W_k$, for $k=1,\dots,K$. On the positive side, these coefficients are spatially smooth when the vector fields $\bfQ_k$ are smooth. However, they exhibit only H\"older continuity in time, and this regularity is not uniform across the sample space $\Omega$. This lack of uniformity further complicates the compactness argument for the density and necessitates the development of new techniques. In particular, there is a need for an extension of techniques that allow to prove extra integrability of the pressure as well as Lions' effective viscous flux identity for rough operators in time.

We first exploit the fact that, due to the divergence-free property of the vector fields $\bfQ_k$, $k=1,\dots,K$, the noise is energy conservative. Consequently, the standard energy inequality provides uniform bounds with respect to the sample space. This allows us to perform the first limit passage in each approximate system, identifying all limit terms except for the pressure. Indeed, the energy estimates alone do not guarantee the strong convergence of the density. To address this, we require higher integrability of the pressure in conjunction with an analysis of the effective viscous flux.
In the transformed system, the effective viscous flux integral takes the form
\begin{align*}
\int_0^{t}\int_{\mt} \big(p(\eta)-(\lambda+2\mu)\Div^{\bfphi} {\bfv}\big),\eta\dx\ds.
\end{align*}
The corresponding compactness argument is based on the use of operators such as $\nabla^{\bfphi}(-\Delta^{\bfphi})^{-1}$, which are time-dependent and non-differentiable in time due to the stochastic nature of $W$. In particular, unlike in standard deterministic approaches, we are unable to compute $\partial_t \nabla^{\bfphi}(-\Delta^{\bfphi})^{-1}$.

To overcome this difficulty, we employ a random splitting of the time interval into short subintervals, freezing the coefficients in time. Working pathwise, we introduce several random quantities that allow us to localize in time and handle the lack of time regularity in the flow uniformly in $\omega$. Unlike in previous studies on stochastic compressible Navier--Stokes equations, the resulting pressure bounds hold pathwise and are therefore $\omega$-dependent. While this initially yields only a converging subsequence that may depend on $\omega$, we make use of the fact that a candidate limit solution was already obtained almost surely via Skorokhod's representation theorem. Thus, no further subsequence extraction is necessary. This pathwise approach suffices to establish the effective viscous flux identity and to identify the limit pressure as required.

The structure of the paper is as follows. In Section~\ref{M} we formulate the problem, introduce the flow transformation and state our main result, Theorem~\ref{thm:main}. In Section~\ref{s:3}, we introduce the approximate system and establish existence and uniqueness of its solutions. The remainder of the paper is devoted to the successive passages to the limit, completing the proof of Theorem~\ref{thm:main}. Specifically, in Section~\ref{s:4} we perform the vanishing viscosity limit and in Section~\ref{s:5} we remove a regularisation parameter as well as the artificial pressure.

\section{Problem formulation and main results}
\label{M}

\subsection{Stratonovich integration}
Let $(\Omega,\mathfrak F,(\mathfrak F_t)_{t\geq0},\mathbb P)$ be a filtered probability space and let $W=(W_k)_{k=1}^K$ be a collection of independent standard $(\mathfrak F_t)$-Wiener processes.
Set $\mathbb Q=(\bfQ_k)_{k=1}^K$.
We define the Stratonovich integrals in \eqref{eq:transport} by means of the It\^{o}--Stratonovich correction.
First of all we define the stochastic integrals
\begin{align*}
&\int_0^t\Div (\vr \mathbb{Q} ) \,\D W=\sum_{k\in I}\int_0^t\Div (\vr \bfQ_k ) \,\D W_k,\\
&\int_0^t\Div (\varrho\bfu\otimes \mathbb{Q} ) \,\D W=\sum_{k=1}^K\int_0^t\Div (\varrho\bfu\otimes \bfQ_k ) \,\D W_k,
\end{align*}
as It\^{o}-integrals on the Hilbert spaces $W^{-\ell,2}(\TN)$ and $W^{-\ell,2}(\TN,R^N)$, $\ell>N/2+1$, respectively. Indeed, if $\vr$ and $\varrho\bfu$
are $(\mathfrak F_t)$-adapted stochastic processes taking values in 
$C_{\rm{weak}}([0,T];L^{q}(\TN))$ and $C_{\rm{weak}}([0,T];L^{q}(\TN,R^N))$, $q > 1$, respectively, and the vector fields  $\bfQ_{k}$ are bounded, then
the It\^{o}-integrals 
\begin{align*}
&\int_0^t\langle\Div (\vr \mathbb{Q} ),\psi\rangle \,\D W=-\sum_{k=1}^K\int_0^t\int_{\TN}\vr \bfQ_k \cdot\nabla_x\psi\dx \,\D W_k,\quad\psi\in W^{\ell,2}(\TN),\\
&\int_0^t\langle\Div (\varrho\bfu\otimes \mathbb{Q} ),\bfpsi\rangle \,\D W=-\sum_{k=1}^K\int_0^t\int_{\TN}\varrho\bfu\otimes \bfQ_k :\nabla_x\bfpsi\dx \,\D W_k,\quad\bfpsi\in W^{\ell,2}(\TN,R^N),
\end{align*}
are well-defined. The corresponding Stratonovich integrals are now defined via the It\^{o}--Stratonovich correction, that is
\begin{align*}
\int_0^t\int_{\TN}\vr \bfQ_k \cdot\nabla_x\psi\dx \,\circ\D W_k&=\int_0^t\int_{\TN}\vr \bfQ_k \cdot\nabla_x\psi\dx \,\D W_k\\&+\frac{1}{2}\Big\langle\Big\langle\int_{\TN}\vr \bfQ_k \cdot\nabla_x\psi\dx,W_k\Big\rangle\Big\rangle_t,\\
\int_0^t\int_{\TN}\varrho\bfu\otimes \bfQ_k :\nabla_x\bfpsi\dx \,\circ\D W_k&=\int_0^t\int_{\TN}\varrho\bfu\otimes \bfQ_k :\nabla_x\bfpsi\dx \,\D W_k\\&+\frac{1}{2}\Big\langle\Big\langle\int_{\TN}\varrho\bfu\otimes \bfQ_k :\nabla_x\bfpsi\dx,W_k\Big\rangle\Big\rangle_t,
\end{align*}
where $\langle\langle\cdot,\cdot\rangle\rangle_t$ denotes the cross variation. We compute now the cross variations by means of \eqref{1.1}--\eqref{1.1b} with \eqref{eq:transport}. Spelling out only the martingale part, we have
\begin{align*}
\left[\int_{\TN}\vr \bfQ_k \cdot\nabla_x\psi\dx\right](t)&=\dots-\sum_{\ell}\int_0^t\int_{\TN}\vr \bfQ_\ell \cdot\nabla_x(\bfQ_k \cdot\nabla_x\psi)\dx \,\D W_\ell,\\
\left[\int_{\TN}\varrho\bfu\otimes \bfQ_k :\nabla_{x}\bfpsi\dx\right](t)&=\dots-\sum_{\ell}\int_0^t\int_{\TN}\varrho\bfu\cdot (\bfQ_\ell \cdot\nabla_{x}(\bfQ_k\cdot\nabla_{x}\bfpsi))\dx \,\D W_\ell,
\end{align*}
where $\psi\in W^{\ell+1,2}(\TN)$ and $\bfpsi\in W^{\ell+1,2}(\TN,R^N)$.
Here the deterministic terms of the equations with quadratic variation zero are irrelevant and hidden. Plugging the previous considerations together we set
\begin{align*}
&\int_0^t\Div_x (\vr \mathbb{Q} ) \,\circ\D W=\sum_{k=1}^K\int_0^t\Div_x (\vr \bfQ_k ) \,\D W_k+\frac{1}{2}\sum_{k=1}^K\int_0^t\Div_x(\bfQ_k \otimes\bfQ_k\nabla_{x}\varrho) \ds,\\
&\int_0^t\Div_x (\varrho\bfu\otimes \mathbb{Q} ) \,\circ\D W=\sum_{k=1}^K\int_0^t\Div_x (\varrho\bfu\otimes \bfQ_k ) \,\D W_k+\frac{1}{2}\sum_{k=1}^K\int_0^t\Div_x(\bfQ_k\otimes\bfQ_k\nabla_{x}(\varrho\bfu))\ds,
\end{align*}
to be understood in $W^{-\ell-1,2}(\TN)$ and $W^{-\ell-1,2}(\TN,R^N)$, respectively.


\subsection{The concept of solution}

We consider the equations on a  time interval $I=[0,T]$ for some $T>0$. For the sake of simplicity, we restrict ourselves to 
deterministic initial data.

\begin{Definition}[Dissipative martingale solution]\label{MD1}
Let $(\varrho_0,\bfq_0)\in L^\gamma(\mt)\times L^{\frac{2\gamma}{\gamma+1}}(\mt)$ with $\varrho_0\geq0$ a.a. be deterministic initial data.
The quantity  $$\big((\Omega,\mf,(\mf_t)_{t\geq0},\prst),\varrho,\bfu,W\big)$$
is called a {\em dissipative martingale solution} to \eqref{1.1}--\eqref{eq:transport} with initial data $(\varrho_0,\bfq_0)$, provided the following holds:
\begin{enumerate}[(a)]
\item $(\Omega,\mf,(\mf_t)_{t\geq 0},\prst)$ is a stochastic basis with a complete right-continuous filtration;
\item $W$ is a $K$-dimensional $(\mf_t)$-Wiener process;
\item the density $\vr$ is non-negative, belongs to the space
$
C_{\rm{weak}}( I; L^{\gamma}(\mt))$
$ \mathbb{P}\mbox{-a.s.}$, and is $(\mf_t)$-adapted;

\item the momentum $\vr\bfu$
belongs to the space
$
 C_{\rm{weak}}( I; L^{\frac{2\gamma}{\gamma+1}}(\mt,R^{N}))
$
\pas\ and is $(\mf_t)$-adapted;
\item the velocity $\bfu$ belongs to $
 L^2(I; W^{1,2}(\mt,R^N))
$
\pas\ and is $(\mathfrak{F}_{t})$-adapted\footnote{Adaptedness of the velocity  is understood in the sense of random distributions, cf. \cite[Chapter~2.8]{BrFeHobook}.};
\item we have $(\varrho(0),(\varrho\bfu)(0))=(\varrho_0,\bfq_0)$ $\mathbb P$-a.s.;
\item the continuity equation holds in the sense that
\begin{align}\label{eq:condW2}
\begin{aligned}
\int_{\mt}\varrho\psi\dx\bigg|_{s=0}^{s=t}&-
\int_0^t\int_{\mt}\varrho \bfu\cdot\nabla_{x}\psi\dxs\\&=
-\int_0^t\int_{\mt}\varrho \mathbb Q\cdot\nabla_x \psi\,\dx\,\dd W\\
&+\frac{1}{2}\sum_{k=1}^K\int_0^t\int_{\mt}\vr\,\Div_x(\bfQ_k \otimes\bfQ_k\nabla_{x}\psi)\dx \ds
\end{aligned}
\end{align}
$\mathbb P$-a.s.  for all $\psi\in C^\infty(\TN)$;
\item the momentum equation holds in the sense that\begin{align}\label{eq:momdW2}
\begin{aligned}
\int_{\mt} \varrho \bfu\cdot \bfpsi\dx\bigg|_{s=0}^{s=t}& -\int_0^t\int_{\mt}\varrho \bfu\otimes \bfu:\nabla_x \bfpsi\dxs
\\
&+\int_0^t\int_{\mt}\mathbb S(\nabla_x \bfu):\nabla_x \bfpsi \dxs-\int_0^t\int_{\mt}
p(\varrho)\,\Div_x \bfpsi \dxs
\\&=-\int_0^t\int_{\mt}\varrho \bfu\otimes\mathbb Q:\nabla_x \bfpsi\,\dx\,\dif W\\&+\sum_{k=1}^K\int_0^t\int_{\mt}\vr\bfu\cdot\Div_x(\bfQ_k\otimes\bfQ_k\nabla_{x}\bfpsi)\dx \ds
\end{aligned}
\end{align} 
$\mathbb P$-a.s. for all $\bfpsi\in C^\infty(\mt,R^{N})$;
\item the energy inequality is satisfied in the sense that
\begin{align} \label{eq:enedW2}
\begin{aligned}
- \int_I &\partial_t \psi \,
\mathscr E \dt+\int_I\psi\int_{\mt}\mathbb S(\nabla_x \bfu):\nabla_x \bfu\dxt\leq
\psi(0) \mathscr E(0)
\end{aligned}
\end{align}
holds $\mathbb P$-a.s. for any $\psi \in C^\infty_c([0, T))$.
Here, we abbreviated
$$\mathscr E(t)= \int_{\TN}\Big(\frac{1}{2} \varrho(t) | {\bfu}(t) |^2 + P(\varrho(t))\Big)\dx,$$
where the pressure potential is given by $P(\varrho)=\frac{a}{\gamma-1}\varrho^\gamma$.
\end{enumerate}

\end{Definition}

\begin{Remark}

It is worth noting that, unlike the field equations \eqref{eq:condW2}, \eqref{eq:momdW2}, the energy inequality 
\eqref{eq:enedW2} does not contain any stochastic integral yielding pathwise uniform bounds in terms of the initial data.
	
	\end{Remark}

\subsection{Flow transformation}\label{s:2.3}
We consider the flow generated by the noise which is given by
\begin{align}\label{eq:flowsde}
\bfphi(x,t)=x-\sum_{k=1}^K\int_0^t\bfQ_k(\bfphi(x,s))\circ \D W_k.
\end{align}
There exists a unique probabilistically strong solution to the SDE \eqref{eq:flowsde} taking values in the class of measure  preserving (since the $\bfQ_k$ are solenoidal) and smooth (as the $\bfQ_k$ are smooth) diffeomorphisms on the torus $\TN$, see \cite{Ku}. We denote by $\bfphi^{-1}$ the inverse flow.
In particular,  it holds
\begin{align}\label{eq:3010a}
\mathbb{E}[\|\bfphi\|^{2}_{C([0,T];C^{1}(\mt,\TN))}]<\infty
\end{align}
and for a.e. $\omega\in\Omega$ and $\alpha\in(0,1/2)$
\begin{align}\label{eq:3010}
\|\bfphi(\omega)\|_{C^{\alpha}([0,T];C^{2}(\mt,\TN))}+\|\bfphi^{-1}(\omega)\|_{C^{\alpha}([0,T];C^{2}(\mt,\TN))}<\infty,
\end{align}
which are the regularity requirements needed in our analysis. The transformation (\ref{eq:flowsde}) is nothing but the stochastic analogue of the Euler--Lagrange coordinate transformation under the flow $\bfQ_k$. It is particularly important in studies of free boundary problems and regular solutions in the critical framework. However, the flow needs to be sufficiently regular in space in order to control the Jacobian of the transformation. In our case, since $\bfQ_k$ is divergence-free, the quantity is simply equal to one \cite{DanMu,Mu}.

With the flow at hand we reformulate equations \eqref{1.1}--\eqref{1.1b} as random PDEs in terms of $\bfv$ and $\eta$ given by
\begin{align*}
\bfv(t,x)=\bfu(t,\bfphi(t,x)),\quad \eta(t,x)=\varrho(t,\bfphi(t,x)).
\end{align*}
Here the composition affects only the spatial variable.
In the new variables we obtain the equations
\begin{align} \label{i1B}
\partial_t \eta + \Div^{\bfphi} (\eta\bfv)  &=0,\\
\label{i2B}
\partial_t (\eta\bfv) + \Div\big(\eta\bfv\otimes\bfv\big) -\Div^{\bfphi}\mathbb S(\nabla^{\bfphi}\bfv) + \nabla^{\bfphi} p(\eta)  &=0,
\end{align}
where the corresponding (random) differential operators are defined as
\begin{align*}
\Div^{\bfphi} \bfpsi=[\Div(\bfpsi\circ\bfphi^{-1})]\circ\bfphi,\quad \nabla^{\bfphi} \psi=[\nabla(\psi\circ\bfphi^{-1})]\circ\bfphi .
\end{align*}
More specifically, the problem 
	\eqref{i1B}, \eqref{i2B}	
	can be written in the form:
\begin{align} 
	\partial_t \eta + \partial_{x_j} \left( \eta A_{j,k} v_k \right) &= 0, \label{ER1} \\ 	
	\partial_t (\eta v_i) + \partial_{x_j} \left( \eta A_{j,k} v_k v_i \right) + A_{l,i} \frac{ \partial }{x_l} p(\eta) &= 
	\mu \frac{\partial }{\partial x_j} \left( A_{j,l} A_{k,l} \frac{\partial v_i }{\partial x_k} \right) + 
	\lambda A_{l,i} \frac{\partial }{\partial x_l} \frac{\partial }{\partial x_j} \left( A_{j,k} v
	_k \right) 
	\label{ER2} \\
	i &=1, \dots, N \nonumber,
\end{align}
where we have denoted
\[
A_{i,j}(t,x) = \frac{ \partial (\varphi^{-1}_i) }{\partial x_j} (t, \bfphi(t,x)). 
\]

We note that
$$
\partial_{i}\bfv^{j}=\partial_{i}(\bfv^{j}\circ\bfphi^{-1}\circ\bfphi) =\partial_{k}^{\bfphi}\bfv^{j}\,\partial_{i}\bfphi^{k}
$$
with the summation over repeated indices. Consequently
\begin{equation}\label{eq:w12}
\|\nabla \bfv\|_{L^{2}(\mt)} \leq \|\nabla^{\bfphi} \bfv\|_{L^{2}(\mt)} \|\nabla\bfphi\|_{L^{\infty}(\mt)},
\end{equation}
which, in view of \eqref{eq:3010a}, can be used to derive a bound for $\bfv$ in $W^{1,2}(\TN,R^{N})$ based on the energy inequality \eqref{eq:enedWB} below.

In the following definition we formulate a notion of solution to the transformed system \eqref{i1B}--\eqref{i2B} in a way that permits to establish an equivalence with Definition \ref{MD1}, see Lemma~\ref{l:equiv} below.  Note, that it is essential to keep track of the noise even in the random PDEs \eqref{i1B}--\eqref{i2B} in order  to guarantee that the transformed density $\eta$ as well as the transformed momentum $\eta\bfv$ are $(\mf_{t})$-adapted.

\begin{Definition}\label{MD1B}
Let $(\varrho_0,\bfq_0)\in L^\gamma(\mt)\times L^{\frac{2\gamma}{\gamma+1}}(\mt)$ with $\varrho_0\geq0$ a.a. be deterministic initial data.
The quantity  $$\big((\Omega,\mf,(\mf_t)_{t\geq0},\prst),\eta,\bfv,W\big)$$
is called a {\em dissipative martingale solution} to \eqref{i1B}--\eqref{i2B} with initial data $(\varrho_0,\bfq_0)$, provided the following holds:
\begin{enumerate}[(a)]
\item $(\Omega,\mf,(\mf_t)_{t\geq 0},\prst)$ is a stochastic basis with a complete right-continuous filtration;
\item $W$ is a $K$-dimensional $(\mf_t)$-Wiener process;
\item the density $\eta$ belongs to the space
$
C_{\rm{weak}}( I; L^{\gamma}(\mt))$
$ \mathbb{P}\mbox{-a.s.}$, is non-negative and is $(\mf_t)$-adapted;

\item the momentum $\eta\bfv$
belongs to the space
$
 C_{\rm{weak}}( I; L^{\frac{2\gamma}{\gamma+1}}(\mt,R^{N}))
$
\pas\ and is $(\mf_t)$-adapted;
\item the velocity $\bfv$ belongs to $
 L^2(I; W^{1,2}(\mt,R^N))
$
\pas\ and is $(\mathfrak{F}_{t})$-adapted;\footnote{Adaptedness of the velocity  is understood in the sense of random distributions, cf. \cite[Chapter~2.8]{BrFeHobook}.}
\item we have $(\eta(0),(\eta\bfv)(0))=(\varrho_0,\bfq_0)$ $\mathbb P$-a.s.;
\item the continuity equation holds in the sense that
\begin{align*}
\begin{aligned}
\int_{\mt}\eta\psi\dx\bigg|_{s=0}^{s=t}&-
\int_0^t\int_{\mt}\eta \bfv\cdot\nabla^{\bfphi}\psi\dxs=0
\end{aligned}
\end{align*}
$\mathbb P$-a.s.  for all $\psi\in C^\infty(\TN)$, where the flow map $\bfphi$ is determined by \eqref{eq:flowsde};
\item the momentum equation holds in the sense that
\begin{align*}
\begin{aligned}
\int_{\mt} \eta \bfv\cdot \bfpsi\dx\bigg|_{s=0}^{s=t} &-\int_0^t\int_{\mt}\eta \bfv\otimes \bfv:\nabla^{\bfphi} \bfpsi\dxs
\\
&+\int_0^t\int_{\mt}\mathbb S(\nabla^{\bfphi} \bfv):\nabla^{\bfphi} \bfpsi \dxs-\int_0^t\int_{\mt}
p(\eta)\,\Div^{\bfphi} \bfpsi \dxs
=0
\end{aligned}
\end{align*} 
$\mathbb P$-a.s. for all $\bfpsi\in C^\infty(\mt,R^{N})$, where the flow map $\bfphi$ is determined by \eqref{eq:flowsde};
\item the energy inequality is satisfied in the sense that
\begin{align} \label{eq:enedWB}
\begin{aligned}
- \int_I &\partial_t \psi \,
\mathscr E^{\bfphi} \dt+\int_I\psi\int_{\mt}\mathbb S(\nabla^{\bfphi} \bfv):\nabla^{\bfphi} \bfv\dxt\leq
\psi(0) \mathscr E^{\bfphi}(0)
\end{aligned}
\end{align}
holds $\mathbb P$-a.s. for any $\psi \in C^\infty_c([0, T))$, where the transformation $\bfphi$ is given by \eqref{eq:flowsde}.
Here, we abbreviated
$$\mathscr E^{\bfphi}(t)= \int_{\TN}\Big(\frac{1}{2} \eta(t) | {\bfv}(t) |^2 + P(\eta(t))\Big)\dx,$$
where the pressure potential is given by $P(\eta)=\frac{a}{\gamma-1}\eta^\gamma$.
\end{enumerate}

\end{Definition}

As mentioned above, the following result permits to switch between solutions to the original stochastic system \eqref{1.1}--\eqref{eq:transport} and the transformed random PDEs  \eqref{i1B}--\eqref{i2B}.

\begin{Lemma}\label{l:equiv}
The quantity  $((\Omega,\mf,(\mf_t)_{t\geq0},\prst),\varrho,\bfu,W)$
is  a  dissipative martingale solution to \eqref{1.1}--\eqref{eq:transport} with initial data $(\varrho_0,\bfq_0)$ if and only if
$((\Omega,\mf,(\mf_t)_{t\geq0},\prst),\eta,\bfv,W))$, where $\eta=\varrho\circ\bfphi$, $\bfv=\bfu\circ\bfphi$ and $\bfphi$ solves \eqref{eq:flowsde}, is  a  dissipative martingale solution to \eqref{i1B}--\eqref{i2B} with initial data $(\varrho_0,\bfq_0)$.
\end{Lemma}

\begin{proof}
First, we observe that since the flow $\bfphi$ as well as its inverse $\bfphi^{-1}$ is $(\mf_t)$-adapted, $(\varrho,\bfu)$ is $(\mf_t)$-adapted if and only if $(\eta,\bfv)$ is $(\mf_t)$-adapted.

The equivalence between the respective  weak formulations of the continuity and momentum equations follows from similar arguments as in Proposition A.1 in \cite{HLP}. The difference lies in the fact that unlike \cite{HLP} we work with deterministic test functions. For instance, if $\varrho$ solves the continuity equation in the sense of \eqref{eq:condW2}, then the evolution of  $\eta$ is obtained from the application of It\^o's product rule in the spirit of Proposition 2.4.2 in \cite{BrFeHobook} applied  to
$$
\dif \int_{\T^{N}} \eta \psi \dx = 
\dif \int_{\T^{N}} \varrho \psi (\bfphi^{- 1}) \dx = \int_{\T^{N}} \dif \varrho \psi (\bfphi^{- 1}) \dx+ \int_{\T^{N}}  \varrho \dif\psi (\bfphi^{- 1}) \dx,
$$
for an arbitrary $\psi\in C^{\infty}(\T^{N})$. Specifically, we work with the duality between $\mathcal{D}'$ and $C^{\infty}$ instead of the
inner product in $L^2$ as done in \cite{BrFeHobook}, while  relying on the smoothness of
the flow for this step. Note that there is no cross variation term because we
consider Stratonovich's integrals. Then we shall use the equation satisfied by $\psi(\bfphi^{-1})$ as formulated in \cite{HLP} to see that the stochastic integrals cancel out. 
For the converse implication, namely, to obtain \eqref{eq:condW2} from the equation for $\eta$, we apply  It\^o's product rule to 
$$
\dif \int_{\T^{N}} \varrho \psi\dx = \dif\int_{\T^{N}} \eta \psi(\bfphi)\dx
$$
and use the equation for $\psi(\bfphi)$ to make the stochastic integral appear.


%

In view of the measure preserving property of the flow of the diffeomorphism $\bfphi$ and \eqref{eq:w12}, all the remaining claims are immediate.
\end{proof}

\subsection{Main result}

Our main goal of this paper is to show the following result.

\begin{Theorem}\label{thm:main}
Let $\bfQ_k \in C^\infty (\TN, R^N)$, ${\rm div}\bfQ_k = 0$,  $k=1,\dots,K$, and let $\gamma > \frac{N}{2}$, $N= 2,3$ be given.
Suppose that $(\varrho_0,\bfq_0)\in L^\gamma(\mt)\times L^{\frac{2\gamma}{\gamma+1}}(\mt)$ with $\varrho_0\geq0$ a.a.  and $\varrho_0^{-1/2}\bfq_0\in L^2(\mt)$ 
are given deterministic initial data. 
Then there exists a dissipative martingale solution to \eqref{i1B}--\eqref{i2B} with the initial condition $(\varrho_{0},\bfq_{0})$ in the sense of Definition \ref{MD1B}. In addition, the continuity equation holds in the renormalised sense, that is we have
\begin{align*}
\begin{aligned}
\int_{\mt}\theta(\eta)\psi\dx\bigg|_{s=0}^{s=t}&-
\int_0^t\int_{\mt}\theta(\eta) \bfv\cdot\nabla^{\bfphi}\psi\dxs=\int_0^t\int_{\mt}\big(\theta(\eta)-\theta'(\eta)\eta\big)\div^{\bfphi}\bfv\,\psi\dxs
\end{aligned}
\end{align*}
$\mathbb P$-a.s. for any $\theta\in BC^2([0,\infty))$.
\end{Theorem}

\subsection{Transformed div-curl lemma}

We record the following version of the div-curl lemma in the transformed coordinates. It is a natural counterpart of \cite[Lemma 3.4]{feireisl1} and relies on the fact that for two operators $A,B$ it holds $(AB)^{\bfphi}=A^{\bfphi}B^{\bfphi}$ as well as $(A^{-1})^{\bfphi}=(A^{\bfphi})^{-1}$.

\begin{Lemma}\label{lem:divcurl}
Let $p,q,r\in(1,\infty)$ such that $\frac{1}{p}+\frac{1}{q}=\frac{1}{r}<1$. Suppose that
$(\bfu_\varepsilon)$ and $(\bfv_\varepsilon)$ are sequences which converge weakly in 
$L^p(\mt, R^N)$ and $L^q(\mt, R^N)$ to $\bfu$ and $\bfv$, respectively. We
 Suppose that $\bfphi,\bfphi_\varepsilon:\mt\rightarrow\mt$ are smooth and measure preserving and invertible functions with $\bfphi_\varepsilon\rightarrow\bfphi$ uniformly in $\mt$.
 Then we have
\begin{align*}
\bfu_\varepsilon\cdot\nabla^{\bfvarphi_\varepsilon}(-\Delta^{\bfvarphi_\varepsilon})^{-1}\Div^{\bfvarphi_\varepsilon}[\bfv_\varepsilon]&-\bfv_\varepsilon\cdot\nabla^{\bfvarphi_\varepsilon}(-\Delta^{\bfvarphi_\varepsilon})^{-1}\Div^{\bfvarphi_\varepsilon}[\bfu_\varepsilon]\\
&\rightarrow \bfu\cdot\nabla^{\bfvarphi}(-\Delta^{\bfvarphi})^{-1}\Div^{\bfvarphi}[\bfv]-\bfv\cdot\nabla^{\bfvarphi}(-\Delta^{\bfvarphi})^{-1}\Div^{\bfvarphi}[\bfu]
\end{align*}
weakly in $L^r(\mt)$, at least for a subsequence.
\end{Lemma}
\begin{proof}
We introduce the sequences 
$\overline\bfu_\varepsilon:=\bfu_\varepsilon\circ\bfphi^{-1}_\varepsilon$ and $\overline\bfv_\varepsilon:=\bfv_\varepsilon
\circ\bfphi^{-1}_\varepsilon$. Clearly they converge weakly in $L^p(\mt)$ and $L^q(\mt)$ too, at least after taking a subsequence. This follows from the fact that $\bfphi_\varepsilon$ is measure-preserving.  
 We must show that the limits coincide with $\overline\bfu:=\bfu\circ\bfphi^{-1}$ and $\overline\bfv:=\bfv\circ\bfphi^{-1}$, respectively.
For $\bfpsi\in\mathcal D(\mt)$ we have
\begin{align*}
\int_{\mt}\overline\bfu_\varepsilon\cdot\bfpsi\dx&=\int_{\mt}\bfu_{\varepsilon}\cdot\bfpsi\circ\bfphi_{\varepsilon}\dx\longrightarrow \int_{\mt}\bfu\cdot\bfpsi\circ\bfphi\dx=\int_{\mt}\overline\bfu\cdot\bfpsi\dx
\end{align*}
and similarly for $\overline\bfv_\varepsilon$.
Since weak limits in $L^p(\mt)$ (respectively $L^q(\mt)$) and in $\mathcal D'(\mt)$ must coincide
the desired convergence follows.
Hence the claim is equivalent to
\begin{align*}
\overline\bfu_\varepsilon\cdot\nabla(-\Delta)^{-1}\Div[\overline\bfv_\varepsilon]&-\overline\bfv_\varepsilon\cdot\nabla(-\Delta)^{-1}\Div[\overline\bfu_\varepsilon]\\
&\rightarrow \overline\bfu\cdot\nabla(-\Delta)^{-1}\Div[\overline\bfv]-\overline\bfv\cdot\nabla(-\Delta)^{-1}\Div[\overline\bfu]
\end{align*}
weakly in $L^r(\mt)$.
This is exactly the original version of \cite[Lemma 3.4]{feireisl1}.
\end{proof}

The rest of the paper is devoted to the proof of Theorem \ref{thm:main}.

\section{The approximate system}\label{s:3}

To construct the solutions we are required to pass through an approximation scheme. Inspired by the theory of deterministic compressible Navier--Stokes equations, we propose the following approximate system.
 Given $l>0$ and $n\in\N$, we define
\begin{equation}\label{sys-apL}
    \begin{array}{l}
       \displaystyle    \partial_t \eta + \div(\bfA^\top_{\bfphi} \eta \bfv ) = 
       \frac{1}{n}\Delta \eta, \ \mbox{where} \ \bfA^\top_{\bfphi} = \Grad \bfphi^{-1}. \\
  \displaystyle  \partial_t \Big(\Big(\frac{1}{n}+[\eta]_l \Big) \bfv \Big) +
    \div\big([\bfA_{\bfphi}^\top\eta \bfv ]_l \otimes \bfv \big) 
    -\frac{1}{2n} [\nabla \eta]_l \nabla \bfv 
    = \div^{\bfphi} \mathbb{S}(\nabla^{\bfphi} \bfv) -\nabla^{\bfphi} p_\delta(\eta),   
    \end{array}
\end{equation}
where $[f]_l=\kappa_l \ast f$ denotes a spatial regularization by means of a convolution with a family of 
	regularizing kernels $\kappa_l$ for $l>0$, meaning
\begin{equation}
[\eta]_l (x) =   \int_{\TN} \kappa_l(y-x) \eta(y)\,\dd y.
\end{equation}
The symbol $p_\delta$ denotes a pressure regularization,
\[
{p_\delta (\eta) = a \eta^\gamma + \delta \eta^2 + \delta \eta^\Gamma, \ \delta > 0,\ 4 \leq \Gamma < 6.}
\]
As $\bfQ_k$ are solenoidal, the transformation $\bfphi$ is volume preserving, and we have
\begin{equation}\label{A-phi}
{\div^{\bfphi} \cdot = \div (\bfA^\top_{\bfphi} \cdot).}  
\end{equation}
The way the regularisation through $[\cdot]_l$ enters equations \eqref{sys-apL} is carefully chosen to ensure energy estimates as we shall see below in \eqref{energy}.

The system (\ref{sys-apL}) is supplemented with approximate initial data 
\begin{equation}\label{in-da-n}
 \eta^{(l)}_0,\bfv^{(l)}_0 \in W^{2-2/p,p}(\mt) \mbox{ \ with \ }  \eta^{(l)}_0 >0,
\end{equation}
where $p> \frac{N+2}{2}$. The initial data are chosen so that $\eta^{(l)}_0 \to \eta_0$ in $L^\Gamma(\mt)$
and $\sqrt{\eta^{(l)}_0} \bfv^{(l)}_0 \to \sqrt{\eta_0} \bfv_0$ in $L^2(\mt, R^N)$.

\medskip 

The idea is to use smoothness of $\bfphi$ in the spatial variable and solve the problem \eqref{sys-apL} pathwise in a semi--deterministic way.

Given a stochastic basis $(\Omega,\mf,(\mf_t)_{t\geq 0},\prst)$ with a complete right-continuous filtration, and a $K$-dimensional $(\mf_t)$-Wiener process $W$, we fix the transformation $\bfphi$ via \eqref{eq:flowsde}. 
Following \cite[Lemma 16]{ChMZ}, we observe that \eqref{sys-apL}, \eqref{in-da-n} 
is a standard parabolic system with coefficients smooth in space and H\" older continuous in time $\prst$-a.s. Accordingly, the problem
admits a local (strong) solution $\eta, \bfv$ unique in the class 
\begin{equation} \label{class}
	\eta, \bfv \in L^p(0,T_{\rm max}; W^{2,p}(\mt) ), \ \partial_t \eta, \partial_t \bfv \in L^p(0,T_{\rm max}; L^{p}(\mt) )  \mbox{ \ as long as \ } p > \frac{N+2}{2}
\end{equation}	 
$\prst$-a.s. Note that the lower bound imposed on the exponent $p$ guarantees boundedness of the approximate density $\eta$ in $L^\infty (\TN \times (0,T_{\rm max}))$. To show $T_{\rm max} = T$, it is enough to establish suitable {\it a priori} bounds independent of $T_{\rm max}$. These follow from the standard 
energy balance obtained via multiplying the second equation in \eqref{sys-apL} by $\bfv$:
\begin{align}
	\frac{1}{2} &\frac{\dd}{\dt} \int_{\mt} \left( \Big(\frac{1}{n} +[\eta]_l \Big) |\bfv|^2 + \frac{\delta}{\Gamma-1}\eta^\Gamma+\delta\eta^2+\frac{a}{\gamma-1}\eta^\gamma  \right) \dx + \int_{\mt} \mathbb{S}(\nabla^{\bfphi} \bfv):\nabla^{\bfphi} \bfv \dx \nonumber \\ &+ \frac{1}{n} \int_{\mt} 
\left( a\gamma \eta^{\gamma - 2} + \delta \Gamma \eta^{\Gamma - 2} + \delta \right)	|\nabla \eta|^2 \dx=0.
\nonumber	
\end{align}

Moreover, we may test $(\ref{sys-apL})_1$ by $\eta$ obtaining
\begin{align*}
	\frac{1}{2}\frac{\dd}{\dt} \int_{\mt} \eta^2 \dx +
	\frac{1}{n}\int_{\mt} |\nabla \eta|^2 \dx&\leq 
	\bigg|\int_{\mt} \eta^2 \div^{\bfphi} \bfv \dx\bigg| \\&\leq \|\eta\|^2_{L^{4}(\mt)}
	\|\nabla^{\bfphi}\bfv\|_{L^{2}(\mt)} \leq C(\delta, \varrho_0,\bfv_0).
\end{align*}

The above estimates are strong enough to guarantee global existence and uniqueness for the system \eqref{sys-apL}, \eqref{in-da-n}, see \cite[Section 2]{ChMZ} for details.
Summarizing, we record the following result.

\begin{Theorem}\label{thm:mainn}
Let $n > 0$, $l > 0$, $\delta > 0$ be fixed. Let the transformation $\bfphi$ be determined by \eqref{eq:flowsde}, where 
$W$ is a K-dimensional $(\mf_t)$-Wiener process defined on a stochastic basis  $(\Omega,\mf,(\mf_t)_{t\geq 0},\prst)$. 
Suppose the initial data $\eta^l_0$, $\bfv^l_0$ belong to the regularity class \eqref{in-da-n}, with $p > \frac{N+2}{2}$.
Then problem \eqref{sys-apL}, \eqref{in-da-n} admits a strong solution $(\eta, \bfv)$ unique in the class \eqref{class} $\prst$-a.s.

In addition the following holds true:
\begin{enumerate}[{\rm (a)}]
\item the density $\eta$ is strictly positive and both $\eta$ and $\bfv$ are $(\mf_t)$-adapted;
\item the energy balance
\begin{align}
	\frac{1}{2} &\frac{\dd}{\dt} \int_{\mt} \left( \Big(\frac{1}{n} +[\eta]_l\Big) |\bfv|^2 + \frac{\delta}{\Gamma-1}\eta^\Gamma+\delta\eta^2+\frac{a}{\gamma-1}\eta^\gamma  \right) \dx + \int_{\mt} \mathbb{S}(\nabla^{\bfphi} \bfv):\nabla^{\bfphi} \bfv \dx \nonumber \\ &+ \frac{1}{n} \int_{\mt} 
	\left( a\gamma \eta^{\gamma - 2} + \delta \Gamma \eta^{\Gamma - 2} + \delta \right)	|\nabla \eta|^2 \dx=0
	\label{energy}
\end{align}
holds $\prst$-a.s.;
\item in particular, we have uniform bounds independent of the approximation parameters:
\begin{align}
\sup_{t \in (0,T)} \| \eta (t, \cdot) \|_{L^\gamma (\TN)} &\leq c(\vr_0, \bfv_0) \\ 
\sup_{t \in (0,T)} \|  [\eta]_l  \bfv (t, \cdot) \|_{L^{\frac{2 \gamma}{\gamma + 1}}(\TN, R^N)} &\leq c(\vr_0, \bfv_0) 	
\end{align}	
$\prst$-a.s.
\end{enumerate}

\end{Theorem}

\section{Vanishing viscosity limit}
\label{s:4}

In this section we let $n\rightarrow\infty$ in system \eqref{sys-apL} to obtain
\begin{align} \label{i1Bl}
\partial_t \eta + \Div^{\bfphi} (\eta\bfv)  &=0,\\
\label{i2Bl}
\partial_t ([\eta]_l\bfv) + \Div \big([A^\top_{\bfphi} \eta\bfv]_l\otimes\bfv\big) -\Div^{\bfphi}\mathbb S(\nabla^{\bfphi} \bfv) + \nabla^{\bfphi} p_\delta(\eta)  &=0.
\end{align}
The solution to \eqref{i1Bl}--\eqref{i2Bl}, supplemented with the initial data
\begin{equation*} 
(\eta (0), ([\eta]_l \bfv) (0)) = (\vr_0, \bfq_0)
\end{equation*}
is understood in the following sense:

\begin{Definition}\label{MD1Bl}
Let $(\varrho_0,\bfq_0)\in L^\Gamma(\mt)\times L^{\frac{2\Gamma}{\Gamma+1}}(\mt)$ with $\varrho_0\geq 0$ a.a.,
$\frac{|\bfq_0|^2}{\vr_0} \in L^1(\TN)$.
The quantity  $$\big((\Omega,\mf,(\mf_t)_{t\geq0},\prst),\eta,\bfv,W)$$
is called a {\em dissipative martingale solution} to \eqref{i1Bl}--\eqref{i2Bl} with initial data $(\varrho_0,\bfq_0)$, provided the following holds:
\begin{enumerate}[(a)]
\item $(\Omega,\mf,(\mf_t)_{t\geq 0},\prst)$ is a stochastic basis with a complete right-continuous filtration;
\item $W$ is a $K$-dimensional $(\mf_t)$-Wiener process;
\item the density $\eta$ belongs to the space
$
C_{\rm{weak}}( I; L^{\Gamma}(\mt))$
$ \mathbb{P}\mbox{-a.s.}$, is non-negative and $(\mf_t)$-adapted;

\item the approximate momentum $[\eta]_l \bfv$
belongs to the space
$
 C_{\rm{weak}}( I; L^{\frac{2\Gamma}{\Gamma+1}}(\mt,R^{N}))
$
\pas\ and is $(\mf_t)$-adapted;
\item the velocity $\bfv$ belongs to $
L^2(I; W^{1,2} (\mt,R^N))
$
\pas\ and is $(\mathfrak{F}_{t})$-adapted;\footnote{Adaptedness of the velocity  is understood in the sense of random distributions, cf. \cite[Chapter~2.8]{BrFeHobook}.}
\item we have $(\eta(0),([\eta]_l \bfv)(0))=(\varrho_0,\bfq_0)$ $\mathbb P$-a.s.;
\item the continuity equation holds in the sense that
\begin{align*}
\begin{aligned}
\int_{\mt}\eta\psi\dx\bigg|_{s=0}^{s=t}&-
\int_0^t\int_{\mt}\eta \bfv\cdot\nabla^{\bfphi}\psi\dxs=0
\end{aligned}
\end{align*}
$\mathbb P$-a.s.  for all $\psi\in C^\infty(\TN)$, where the flow map $\bfphi$ is determined by \eqref{eq:flowsde};
\item the momentum equation holds in the sense that
\begin{align*}
\begin{aligned}
\int_{\mt} [\eta]_l \bfv&\cdot \bfpsi\dx\bigg|_{s=0}^{s=t} -\int_0^t\int_{\mt}[\bfA_{\bfphi}^\top\eta \bfv]_l\otimes \bfv:\nabla \bfpsi\dxs
\\
&+\int_0^t\int_{\mt}\mathbb S(\nabla^{\bfphi} \bfv):\nabla^{\bfphi} \bfpsi \dxs-\int_0^t\int_{\mt}
p_\delta(\eta)\,\Div^{\bfphi} \bfpsi \dxs
=0
\end{aligned}
\end{align*} 
$\mathbb P$-a.s. for all $\bfpsi\in C^\infty(\mt,R^{N})$, where the flow map $\bfphi$ is determined by \eqref{eq:flowsde};
\item the energy inequality is satisfied in the sense that
\begin{align} \label{eq:enedW}
\begin{aligned}
- \int_I &\partial_t \psi \,
\mathscr E^{\bfphi}_{l,\delta} \dt+\int_I\psi\int_{\mt}\mathbb S(\nabla^{\bfphi} \bfv):\nabla^{\bfphi} \bfv\dxt\leq
\psi(0) \mathscr E_{l,\delta}^{\bfphi}(0)
\end{aligned}
\end{align}
holds $\mathbb P$-a.s. for any $\psi \in C^\infty_c([0, T))$, where the flow map $\bfphi$ is determined by \eqref{eq:flowsde}.
Here, we abbreviated
$$\mathscr E^{\bfphi}_{l,\delta}(t)= \int_{\mt}\Big(\frac{1}{2} [\eta(t)]_l | {\bfv}(t) |^2 + P_\delta(\eta(t))\Big)\dx,$$
where the pressure potential is given by $P_\delta(\eta)=\frac{\delta}{\Gamma-1}\eta^\Gamma+\delta\eta^2+\frac{a}{\gamma-1}\eta^\gamma$.
\end{enumerate}

\end{Definition}
We are going to prove the following result.

\begin{Theorem}\label{thm:mainl}
Suppose that $(\varrho_0,\bfq_0)\in L^\Gamma(\mt)\times L^{\frac{2\Gamma}{\Gamma+1}}(\mt)$ with $\varrho_0\geq0$ a.a.  and $\varrho_0^{-1/2}\bfq_0\in L^2(\mt)$ 
are given deterministic initial data. Then there exists a dissipative martingale solution to \eqref{i1Bl}--\eqref{i2Bl} with the initial condition $(\varrho_{0},\bfq_{0})$ in the sense of Definition \ref{MD1Bl}. In addition, the continuity equation holds in the renormalised sense, that is we have
\begin{align}\label{eq:condrenlWB}
\begin{aligned}
\int_{\mt}\theta(\eta)\psi\dx\bigg|_{s=0}^{s=t}&-
\int_0^t\int_{\mt}\theta(\eta) \bfv\cdot\nabla^{\bfphi}\psi\dxs=\int_0^t\int_{\mt}\big(\theta(\eta)-\theta'(\eta)\eta\big)\div^{\bfphi}\bfv\,\psi\dxs
\end{aligned}
\end{align}
$\mathbb P$-a.s. for any $\theta\in BC^2([0,\infty))$.
\end{Theorem}

Given initial data $(\varrho_0,\bfq_0)\in L^\Gamma(\mt)\times L^{\frac{2\Gamma}{\Gamma+1}}(\mt)$ with $\varrho_0\geq0$ a.a. and $\bfq_0$ such that $\varrho_0^{-1/2}\bfq_0\in L^2(\mt)$ we can regularize them to obtain a sequence $(\varrho_0^n,\bfu_0^n)$ of class $C^2(\mt)$ such that $\varrho_0^n >0$ and
\begin{align*}
\int_{\mt}\bigg(P_\delta(\varrho_0^n)+\frac{1}{2}\frac{|\bfq_0^n|}{\varrho_0^n}\bigg)\dx\rightarrow \int_{\mt}\bigg(P_\delta(\varrho_0)+\frac{1}{2}\frac{|\bfq_0|}{\varrho_0}\bigg)\dx
\end{align*}
as $n\rightarrow \infty$.
By Theorem \ref{thm:mainn}, for any fixed $n\in \N$, there is a strong solution $(\eta_n,\bfv_n)$ to system \eqref{sys-apL}. 
They are defined on a stochastic basis $(\Omega,\FF,(\FF_t)_{t\geq0},\p)$ with associated Wiener process $W$. 	
On account of the energy balance \eqref{energy} we obtain immediately the deterministic bounds 
\begin{align}\label{eq:enn}
\begin{aligned}
\sup_{I}\|\eta_n\|_{L^\Gamma(\mt)}^\Gamma
+\sup_{I}\left\| \left( \frac{1}{\sqrt{n}} + \sqrt{[\eta_n]_l} \right)  \bfv_n\right\|^2_{L^2(\mt; R^N)}&\leq\,C(\varrho_0,\bfq_0)\\
\frac{1}{n}\int_I\Big(\big\|\nabla|\eta_n|^{\frac{\Gamma}{2}}\big\|^2_{L^2(\mt)}\dt+\big\|\nabla\eta_n\big\|^2_{L^2(\mt)}\Big)\dt
+\int_I\|\nabla^{\bfphi}\bfv_n\|^2_{L^2(\mt)}\dt&\leq\,C(\varrho_0,\bfq_0),
\end{aligned}
\end{align}
which are uniform with respect to $n$.

\subsection{Stochastic compactness}

We introduce the pathspace
\[
\mathfrak{X} = \mathfrak{X}_{\eta}\times \mathfrak{X}_{\eta \bfv} \times \mathfrak{X}_{\bfv} \times \mathfrak{X}_p 
\times \mathfrak{X}_{\bfphi}\times\mathfrak{X}_{W} \times \mathfrak{X}_{\nu},
\]
where 
\begin{align*}
\mathfrak{X}_{\eta} = C_{\rm{weak}}\big([0,T]; L^\Gamma(\mt)\big)&,\quad 
\mathfrak{X}_{\eta \bfv} = C_{\rm{weak}}([0,T]; L^{\frac{2\Gamma}{\Gamma + 1}}(\mt,R^N)),\\
\mathfrak{X}_{\bfv} = \left(L^{2}(0,T; W^{1,2}(\mt,R^N)),w\right)&,\quad 
\mathfrak{X}_p = (L^\infty(0,T;  \mathcal{M}^+ (\mt)),w^*), \\ 
\mathfrak{X}_{\bfphi}= C^{\alpha}([0,T];C^{2}(\mt,\mt))&,\quad 
\mathfrak{X}_{W} = C([0,T];R^{K}),\\
\mathfrak{X}_{\nu} = (L^\infty((0,T) \times \mt&; {\rm Prob}( R^{1+N+N^2})),w^*),
\end{align*}
for some $\alpha\in (0,1/2)$.
%
 Similarly to \cite[Section~4.5]{BrFeHobook}
it follows that the
family of joint laws
\[
\left\{\mathcal{L} \left[ \eta_n,
[\eta_n]_l \bfv_n,\bfv_n, p_\delta (\eta_n) , \bfphi, {W}, \delta_{[\eta_n, \bfv_n,\nabla\bfv_n]}\right] ;\,n\in\mathbb N\right\}
\]
is tight on $\mathfrak X$. More precisely, using \eqref{eq:w12}, \eqref{eq:3010a} and \eqref{eq:enn} we deduce
\begin{equation} \label{new}
\mathbb{E}\left[\int_{I}\|\nabla \bfv_{n}\|_{L^{2}}^{2}\dt\right]\leq\mathbb{E}\left[\|\bfphi\|_{C([0,T];C^{1}(\TN,\TN))}^{2}\int_{I}\|\nabla^{\bfphi} \bfv_{n}\|_{L^{2}}^{2}\dt\right] \leq C
\end{equation}
uniformly in $n$, which implies the desired tightness of $\nabla \bfv_{n}  
\ \mbox{in}\ (L^2((0,T) \times \TN, R^{N \times N}$ ),w) , $n\in\N$. 
This bound, together with the first estimate in \eqref{eq:enn}, and 
a generalized version of Korn's inequality (see e.g.    \cite[Chapter 11, Theorem 11.23]{FeNo6A}), yields tightness of 
$\bfv_n$ in $\mathfrak{X}_v$.
Tightness of the other quantities follows the arguments from \cite[Section~4.5]{BrFeHobook}.

In view of Jakubowski's version of Skorokhod's representation theorem \cite{jakubow} (see also  \cite[Section 2.8]{BrFeHobook} for property c), we have the following result.

\begin{Proposition} \label{skorokhod}
There exists a complete probability space  $(\Tilde{\Omega}, \Tilde{\mathfrak F},\Tilde{\p})$ with $\mathfrak{X}$-valued random variables  $[\tilde\eta_n,
[\tilde\eta_n]_l \tilde\bfv_n,\tilde\bfv_n,  p_\delta( \tilde \eta_n ) , \tilde\bfphi_{n}, \tilde{W}_{n}, \tilde\nu_n]$, $n\in\N$, and $[\tilde\eta,
[\tilde\eta]_l \tilde\bfv,\tilde\bfv,  \overline{p_\delta(\tilde \eta)} , \tilde\bfphi, \tilde{W}, \tilde\nu]$
such that
\begin{itemize}
    \item [(a)] For all $n \in \N$ the laws of
    $[\tilde\eta_n,
[\tilde\eta_n]_l \tilde\bfv_n,\tilde\bfv_n, p_\delta( \tilde \eta_n ) , \tilde\bfphi_{n}, \tilde{W}_{n}, \tilde\nu_{n}]$ \\ and $[ \eta_n
,[\eta_n]_l \bfv_n,\bfv_n,  p_\delta( \tilde \eta_n ) , \bfphi, {W}, \delta_{[\eta_n, \bfv_n,\nabla\bfv_n]}]$ on $\mathfrak{X}$ coincide;
    \item[(b)] The convergence 
$[\tilde\eta_n,[\tilde\eta_n]_l \tilde\bfv_n,\tilde\bfv_n,  p_\delta( \tilde \eta_n ) , \tilde\bfphi_{n}, \tilde{W}_{n},  \tilde\nu_{n}]\to [\tilde\eta,
[\tilde\eta ]_l \tilde\bfv,\tilde\bfv, \overline{p_\delta(\tilde \eta)} ,  \tilde\bfphi, \tilde{W},  \tilde\nu]$ holds true  in the topology of $\mathfrak{X}$, specifically, 
    \begin{equation}\label{49}
        \begin{cases}
        \Tilde{\eta}_n \to \Tilde{\eta} \,\, \text{in }\,\,  C_{\rm{weak}}\big([0,T]; L^\Gamma(\mt)\big),\\
  [\Tilde{\eta}_n ]_l\Tilde{\bfv}_n \to [\Tilde{\eta}]_l \Tilde{\bfv}  \,\, \text{in }\,\,  C_{\rm{weak}}\big([0,T]; L^{2\frac{\Gamma}{\Gamma+1}}(\mt)\big),\\
   \Tilde{\bfv}_n \rightharpoonup \Tilde{\bfv} \,\, \text{in }\,\,  L^2\big(0,T;W^{1,2}(\mt)),\\
   p_\delta( \tilde \eta_n ) \to  \overline{p_\delta(\tilde \eta)} \ \mbox{in}\ 
   \left( L^\infty(0,T; \mathcal{M}(\TN)), w^* \right) \\
   \tilde\bfphi_{n}\to\tilde\bfphi\,\, \text{in} \,\,  C^{\alpha}([0,T];C^{2}(\TN,\TN)),\\
         \Tilde{W}_{n} \to \Tilde{W} \,\, \text{in} \,\,  C([0,T];R^{K}),\\
 \tilde\nu_{n}\to \tilde\nu \,\, \text{in} \,\,(L^\infty((0,T) \times \mt; {\rm Prob}( R^{1+N+N^2})),w^*),
        \end{cases}
    \end{equation}
    $\Tilde{\p}$-a.s.;
    \item[(c)]  For any Carath\'eodory function $H =H(t,x,\eta,\bfw, \bfW)$ where $(t,x)\in (0,T)\times\TN$, $ (\eta,\bfw,\bfW)\in R^{1+N+N^2}$, satisfying for some $q_1,q_2>0$ the growth condition 
    \[
    H(t,x,\eta,\bfv, \bfW)\lesssim 1 +|\eta|^{q_1}+|\bfw|^{q_2}+|\bfW|^{q_2}
    \]
    uniformly in $(t,x)$, we denote by $\overline{H(\tilde\eta,\tilde{\bfv}, \nabla\tilde{\bfv})}(t,x) =\langle \tilde\nu(t,x),H \rangle$. Then
    \[
    H(\tilde\eta_{n},\tilde{\bfv}_n, \nabla\tilde{\bfv}_n) \rightharpoonup \overline{H(\tilde\eta,\tilde{\bfv}, \nabla\tilde{\bfv})}\quad\text{in}\,\,L^{k}((0,T)\times\TN)
    \]
    holds $\tilde{\p}$-a.s. as  $n \to \infty$ provided $1 < k\leq\frac{{\Gamma}}{q_1}\wedge \frac{2}{q_2}$;
    \item[(d)] $\tilde\bfphi_{n}$, $n\in\N$, and $\tilde\bfphi$ are flows of measure preserving $C^{\infty}$-diffeomorhisms on $\TN$, $\tilde\bfphi_{n}$ is associated to $\tilde{W}_{n}$ via \eqref{eq:flowsde} and similarly $\tilde\bfphi$ solves \eqref{eq:flowsde} with $\tilde W$. Furthermore, the inverse flows satisfy
    $$
    \tilde\bfphi_{n}^{-1}\to \tilde\bfphi^{-1}  \,\, \text{in} \,\,  C^{\alpha}([0,T];C^{2}(\TN,\TN))
    $$
    $\Tilde{\p}$-a.s. 
  
\end{itemize}
\end{Proposition}

To guarantee adaptedness of random variables and to ensure that the stochastic integral is well-defined on the new probability space we introduce filtrations for correct measurability. 
Let $\tilde{\FF}_t$ and $\tilde{\FF}_{t}^{n}$ be the $\tilde{\p}$-augmented filtration of the corresponding random variables from Proposition \ref{skorokhod}, i.e.
\begin{align*}
\tilde{\mathfrak F}_t &= \sigma \bigg(\sigma[\rr\tilde\varrho]\cup\sigma_t[\tilde\bfv]\cup\sigma[\rr\tilde W]\cup \{ \mathcal{N} \in \tilde{\mathfrak F};\tilde{\p}(\mathcal{N})=0\}\bigg), t\geq 0,\\
\tilde{\mathfrak F}_{t}^{n}&=\sigma \bigg(\sigma[\rr\tilde\varrho_n]\cup\sigma_t[\tilde\bfv_n]\cup\sigma[\rr\tilde W^n]\cup\{ \mathcal{N} \in \tilde{\mathfrak F};\tilde{\p}(\mathcal{N})=0\}\bigg), t\geq 0.
\end{align*}
Here $\rr$ denotes the restriction operator to the interval $[0,t]$ on the path space and $\sigma_t$ denotes the history of a random distribution\footnote{For a  random distribution $\bfV$, the family of $\sigma$-fields $(\sigma_{t}[\bfV])_{t\geq 0}$  given by
 \begin{equation*}
     \sigma_t[\bfV]:= \bigcap_{s>t}\sigma\left(\bigcup_{\varphi\in C_c^{\infty}(Q;R^N)}\{\langle \bfV, \varphi \rangle <1 \}\cup \{\mathcal N\in \FF, \p(\mathcal N)=0\} \right)
 \end{equation*}
 is called the history of $\bfV$. In fact, any random distribution is adapted to its history, see \cite[Chap. 2.2]{BrFeHobook}.}.

Since
\eqref{sys-apL} is a random PDE and no stochastic integral appears one easily checks that all quantities are measurable on the paths space. Hence we conclude
that \eqref{sys-apL} continues to holds on the new probability space and the same applies to the energy balance \eqref{energy}. In particular, we have
\begin{align*}
\int_{\mt} [\tilde\eta_n]_l \tilde\bfv_n&\cdot \bfpsi\dx\bigg|_{s=0}^{s=t} -\int_0^{t}\int_{\mt}[\bfA^\top_{\tilde\bfphi_n}\tilde\eta_n \tilde\bfv_n]_l\otimes \tilde\bfv_n:\nabla \bfpsi\dxs
\\
&+\int_0^{t}\int_{\mt}\mathbb S(\nabla^{\tilde\bfphi_n} \tilde\bfv_n):\nabla^{\tilde\bfphi_n} \bfpsi \dxs-\int_0^{t}\int_{\mt}
p_\delta(\tilde\eta_n)\,\Div^{\tilde\bfphi_n} \bfpsi \dxs\\
&=\frac{1}{2n}\int_0^{t}\int_{\mt}[\nabla\tilde\eta_n]_l\nabla\tilde\bfv_n\cdot\bfpsi\dxs
\end{align*} 
$\Tilde{\mathbb P}$-a.s. for all $\bfpsi\in C^\infty(\mt,
R^{N})$. 
By virtue of the convergences \eqref{49}, we may pass to the limit in all terms 
 obtaining
\begin{align}\label{eq:momdWB}
\begin{aligned}
\int_{\mt} [\tilde\eta]_l \tilde\bfv&\cdot \bfpsi\dx\bigg|_{s=0}^{s=t} -\int_0^{t}\int_{\mt}[\bfA_{\tilde\bfphi}^\top\tilde\eta \tilde\bfv]_l\otimes \tilde\bfv:\nabla \bfpsi\dxs
\\
&+\int_0^{t}\int_{\mt}\mathbb S(\nabla^{\tilde\bfphi} \tilde\bfv):\nabla^{\tilde\bfphi} \bfpsi \dxs- \int_0^{t} \left( \int_{\mt}
\Div^{\tilde\bfphi}\bfpsi\, {\rm d} \overline{p_\delta(\tilde\eta)} \right) \ds
=0
\end{aligned}
\end{align} 
$\Tilde{\mathbb P}$-a.s. for all $\bfpsi\in C^\infty(\mt,R^{N})$. Here the function $\tilde{\bfphi}$ is related to $\tilde W$ from Proposition \ref{skorokhod} in the same way the function ${\bfphi}$ is related to the original Wiener process $W$. 
We recall that we have only the bound $p_{\delta}(\tilde\eta_{n})$ in $L^\infty(0,T; L^1(\TN))$ hence the limit $\overline{p_\delta(\tilde\eta)}$ 
must be interpreted as a (non-negative) Radon measure in the spatial variable. In what follows, we establish with the help of a pressure estimate that it has the desired structure $\overline{p_\delta(\tilde\eta)}=p_{\delta}(\tilde\eta)$. 

%

\subsection{Pressure estimates}\label{sec:pressure}

In order to obtain higher integrability of the pressure, we need a modification of the standard method based on  freezing  the coefficients 
depending on the test functions
	in time. We apply this method pathwise, meaning, throughout this section we always work with a fixed $\tilde\omega\in\tilde\Omega$. More precisely, we fix $\tilde\omega\in\tilde\Omega$ such that the convergences in \eqref{49} holds true. It is important to note that many of the below defined objects are random, i.e. they depend on $\tilde\omega$. We always specify this fact in the corresponding definitions, but for notational convenience we do not always spell out this dependence in subsequent computations.

Define
$$
L=L(\tilde\omega):=\|\tilde\bfphi(\tilde\omega)\|_{C^{\alpha}([0,T],C^{2}(\TN,\TN))}\vee\|\tilde\bfphi^{-1}(\tilde\omega)\|_{C^{\alpha}([0,T],C^{2}(\TN,\TN))}.
$$
Since $\tilde\bfphi_{n}(\tilde\omega)\to \tilde
\bfphi(\tilde\omega)$ and $\tilde\bfphi^{-1}_{n}(\tilde\omega)\to \tilde
\bfphi^{-1}(\tilde\omega)$ in $C^{\alpha}([0,T],C^{2}(\TN,\TN))$, there exists $n_{0}=n_{0}(\tilde\omega)\in\N$ so that for all $n\geq n_{0}(\tilde\omega)$ it holds
$$
\|\tilde\bfphi_{n}(\tilde\omega)\|_{C^{\alpha}([0,T],C^{2}(\TN,\TN))}\vee\|\tilde\bfphi_{n}^{-1}(\tilde\omega)\|_{C^{\alpha}([0,T],C^{2}(\TN,\TN))}\leq L(\tilde\omega)+1.
$$
Next, we find a (random) partition of the time interval $[0,T]$ with the following property: for a given $\varkappa=\varkappa(\tilde\omega) >0$\footnote{Eventually, we will choose $\varkappa$ depending on $L$.} we  find $M=M(\tilde\omega)=M(L(\tilde\omega),\varkappa(\tilde\omega))\in\N$ and
\begin{equation}\label{times}
    0=t_{0}<t_{1}=t_1(\tilde\omega)<t_{2}=t_2(\tilde\omega)<\cdots<t_M=T
\end{equation}
such that for all $i\in\{0,\dots,M(\tilde\omega)-1\}$ and all $n\geq n_{0}(\tilde\omega)$
\begin{equation}
    \|\bfA_{\tilde\bfphi_{{n}}({\tilde\omega},t)}-\bfA_{\tilde\bfphi_{{n}}({\tilde\omega},t_i(\tilde\omega))}\|_{L^\infty(\mt)} \leq \varkappa(\tilde\omega) \mbox{ \ \ for all \ \ }  t \in (t_i(\tilde\omega),t_{i+1}(\tilde\omega)).
\end{equation}
Finally, for each $i$ we define $\tilde\Phi_{i,n}=\tilde\Phi_{i,n}(\tilde\omega)$ as a solution to
\begin{equation}\label{Bog}
    \div^{{\tilde\bfphi_{n}}(\tilde\omega,t_i(\tilde\omega))} \tilde\Phi_{i,n}(\tilde\omega) = \tilde\eta_n(\tilde\omega) - m_0,
\end{equation}
where $m_0:= \int_{\mt}\varrho_0\dx$. Specifically, we set $\tilde\Phi_{i,n}=-\nabla^{\tilde\bfphi_{n}(t_i)} (-\Delta^{\tilde\bfphi_{n}(t_i)})^{-1}(\tilde\eta_n-m_0)$. The freezing of time in the operator in (\ref{Bog}) is needed in order to be able to differentiate in time (note that the coefficients arising from $\tilde\bfphi_n$ via \eqref{eq:flowsde} are not differentiable in time).

In the sequel, we continue working with our fixed $\tilde\omega$ only, but we no longer write it down explicitly.
The standard elliptic theory yields
\begin{align}\label{eq:2311}
\|\nabla\tilde\Phi_{i,n}\|_{L^p(\mt, R^N)}&\leq\,c(L)\|\tilde\eta_n\|_{L^p(\mt)}
\end{align}
uniformly in $n$ for $t \in (t_i,t_{i+1})$, provided $\tilde\eta(t,\cdot)\in L^p(\mt)$. 

Next, we compute
\begin{align} \label{ppp} 
\int_{t_i}^{t_{i+1}}\int_{\mt}p_\delta(\tilde\eta_n)(\tilde\eta_n-m_0)\dxs&=\int_{t_i}^{t_{i+1}}\int_{\mt}p_\delta(\tilde\eta_n)\div^{\tilde\bfphi_{n}(t_i)}\tilde \Phi_{i,n} \dxs\\ \nonumber
&=\int_{t_i}^{t_{i+1}}\int_{\mt}p_\delta(\tilde\eta_n)\div^{\tilde\bfphi_{n}} \tilde\Phi_{i,n} \dxs\\ \nonumber &+\int_{t_i}^{t_{i+1}}\int_{\mt}p_\delta(\tilde\eta_n)\big(\div^{\tilde\bfphi_{n}(t_i)}-\div^{\tilde\bfphi_{n}}\big) \tilde\Phi_{i,n} \dxs.
\end{align}
The second term on the right hand side is controlled by
\begin{align*}
\sup_{(t_i,t_{i+1})}\|\bfA_{\tilde\bfphi_{n}}-\bfA_{\tilde\bfphi_{n}(t_i)}\|_{L^\infty(\mt)}&\int_{t_i}^{t_{i+1}}\|p_\delta(\tilde\eta_n)\|_{L^{\frac{\Gamma+1}{\Gamma}}(\mt)}\|\nabla \tilde\Phi_{i,n} \|_{L^{\Gamma+1}(\mt)}\ds\\
&\leq\,\varkappa c(L)  \int_{t_i}^{t_{i+1}} \|\tilde\eta_n\|_{L^{\Gamma+1}(\mt)}^{\Gamma+1}\ds.
\end{align*}
Since $L>0$ is fixed, we may choose $\varkappa$  small enough so  that this term can be absorbed by the left--hand side of \eqref{ppp}. 
The first term on the right-hand side of \eqref{ppp} can be controlled by means of the momentum equation:
\begin{align}\label{eq:2711}
\int_{t_i}^{t_{i+1}}\int_{\mt}
p_\delta(\tilde\eta_n)\,\Div^{\tilde\bfphi_{n}} \tilde\Phi_{i,n} \dxs
&=\int_{\mt}\big(\tfrac{1}{n}+ [\tilde\eta_n]_l\big) \tilde\bfv_n\cdot \tilde\Phi_{i,n}\dx\bigg|_{s=t_i}^{s=t_{i+1}}\\\nonumber
&+\int_{t_i}^{t_{i+1}}\int_{\mt}\big(\tfrac{1}{n}+ [\tilde\eta_n]_l\big) \tilde\bfv_n\cdot\partial_t\tilde\Phi_{i,n}\dxt\\
& -\int_{t_i}^{t_{i+1}}\int_{\mt}[\bfA^\top_{\tilde\bfphi_{n}}\tilde\eta_n \tilde\bfv_n]_l\otimes \tilde\bfv_n:\nabla\tilde\Phi_{i,n}\dxs\nonumber
\\
&+\int_{t_i}^{t_{i+1}}\int_{\mt}\mathbb S(\nabla^{\tilde\bfphi_{n}} \tilde\bfv_n):\nabla^{\tilde\bfphi_{n}} \tilde\Phi_{i,n} \dxs\nonumber\\
&-\tfrac{1}{2n}\int_{t_i}^{t_{i+1}}\int_{\mt}[\nabla\tilde\eta_n]_l\nabla\tilde\bfv_n\cdot\tilde\Phi_{i,n} \dxs\nonumber\\
&=:(\tt{I})^i+\dots+(\tt{V})^i.\nonumber
\end{align} 
By \eqref{eq:2311} and Poincar\' e's inequality we have 
\begin{align*}
(\tt{I})^i
&\leq \sup_{(t_{i},t_{i+1})}\frac{1}{\sqrt{n}} \big\| \tfrac{1}{\sqrt{n}}\tilde\bfv_n\big\|_{L^2(\mt)}
    \|\tilde\Phi_{i,n}\|_{L^2(\mt)}\\
    &+ \sup_{(t_{i},t_{i+1})}\| \sqrt{[\tilde\eta_n]_l} \tilde\bfv_n\|_{L^2(\mt)} \|\sqrt{[\tilde\eta_n]_l}\|_{L^4(\mt)}\|\tilde\Phi_{i,n}\|_{L^4(\mt)}\\
&\leq\,c(L) \sup_{(t_{i},t_{i+1})}\frac{1}{\sqrt{n}} \big\| \tfrac{1}{\sqrt{n}}\tilde\bfv_n\big\|_{L^2(\mt)}
   \|\tilde\eta_n\|_{L^2(\mt)}\\&+ c(L)\sup_{(t_{i},t_{i+1})}\| \sqrt{[\tilde\eta_n]_l} \tilde\bfv_n\|_{L^2(\mt)} \|\tilde\eta_n\|_{L^2(\mt)}^{\frac{1}{2}}\|\tilde\eta_n\|_{L^4(\mt)}.
\end{align*}
The right-hand side is uniformly bounded by \eqref{eq:enn} since $\Gamma\geq 4$. Here, we have used that \eqref{eq:enn}
as well as \eqref{new} hold on the new probability space due to Proposition  \ref{skorokhod}.

For the term $(\tt II)^i$ we note that 
\begin{align}
\begin{aligned}
   \partial_t \tilde\Phi_{i,n}&=-\nabla^{\tilde\bfphi_{n}(t_i)} (-\Delta^{\tilde\bfphi_{n}(t_i)})^{-1}\partial_t \tilde\eta_n\\
   &= \nabla^{\tilde\bfphi_{n}(t_i)} (-\Delta^{\tilde\bfphi_{n}(t_i)})^{-1}\div^{\tilde\bfphi_{n}}(\tilde\eta_n \tilde \bfv_n) - \tfrac{1}{n}
    \nabla^{\tilde\bfphi_{n}(t_i)} (-\Delta^{\tilde\bfphi_{n}(t_i)})^{-1}\Delta \tilde\eta_n
\end{aligned}
\end{align}
in $(t_i,t_{i+1})$ using \eqref{sys-apL}. By elliptic theory we have
\begin{align*}
   \|\partial_t \tilde\Phi_{i,n}\|_{L^2(\mt)}\leq\,c(L)\|\tilde\eta_n  \tilde\bfv_n\|_{L^2(\mt)}+\,c(L)\big\|\tfrac{1}{n}\nabla\tilde\eta_n\big\|_{L^2(\mt)}
\end{align*}
in $(t_i,t_{i+1})$.
Thanks to the choice of large $\Gamma$ we are ensured that $\tilde\eta_n \in L^\infty_t L^3_x$ uniformly by \eqref{eq:enn}, so we deduce that
\begin{equation}
   \tilde \eta_n \tilde\bfv_n \in L^2(I;L^2(\mt))
    \mbox{ \ \ and \ \ }
    \frac{1}{\sqrt{n}} \nabla \tilde\eta_n\in L^2(I;L^2(\mt)) \;\;(\mbox{uniformly in $n$})
\end{equation}
thus    $\partial_t \tilde\Phi_{i,n} \in L^2(t_i,t_{i+1};L^2(\TN))$ uniformly as well.

Similarly, we have
$(\frac{1}{n}+[\tilde\eta_n]_l ) \tilde\bfv_n \in L^2(I;L^2(\mt))
   $ uniformly such that $({\tt II})^i$ is uniformly bounded.

We proceed by
\begin{align*} 
({\tt III})^i& = 
    - \int_{t_i}^{t_{i+1}} 
    \int_{\TN}  [\bfA^\top_{\tilde\bfphi_n} \tilde\eta_n\tilde \bfv_n]_l \otimes \tilde\bfv_n :\nabla\tilde\Phi_{i,n} \dx\ds
    \\ 
    &\leq \,c(L) \|\tilde\bfv_n\|^2_{L^2(t_i,t_{i+1};L^6(\mt))}
    \|\tilde\eta_n\|^2_{L^\infty(t_i,t_{i+1};L^{3}(\mt))} 
\end{align*}
using again \eqref{eq:2311} (with $p=3$). This is again uniformly bounded by \eqref{eq:enn} because of the Sobolev embedding 
$W^{1,2} \hookrightarrow L^6$, $N \leq 3$.  We can argue similarly for $({\tt IV})^i$.

Next, since $\Gamma > N$, and in view of \eqref{eq:w12}, we can bound
\begin{align*}({\tt V})^i&
    \leq \frac{1}{\sqrt{n}}\int_{t_i}^{t_{i+1}} \big\|\tfrac{1}{\sqrt{n}} \nabla\tilde\eta_n\big\|_{L^2(\mt)} \|\nabla\tilde\bfv_n\|_{L^2(\mt)} \|\tilde\Phi_{i,n}\|_{L^\infty(\mt)}\ds\\
   &\leq \frac{c(L)}{\sqrt{n}}\int_{t_i}^{t_{i+1}} \big\|\tfrac{1}{\sqrt{n}} \nabla\tilde\eta_n\big\|_{L^2(\mt)} \|\nabla^{\tilde\bfphi_{n}}\tilde\bfv_n\|_{L^2(\mt)} \|\tilde\eta_{n}\|_{L^\Gamma(\mt)}\ds
\end{align*}
using a generalized version of Korn's inequality (see e.g.    \cite[Chapter 11, Theorem 11.23]{FeNo6A}) and  \eqref{eq:2311} (with $p=\Gamma$).  This is again uniformly bounded by \eqref{eq:enn}.

Summing up the previous estimates and using $\eta^\Gamma \leq\,\delta^{-1} p_\delta(\eta)$ we obtain
\begin{equation}
    \int_{t_i}^{t_{i+1}} \int_{\mt} \tilde\eta_n^{\Gamma+1} \dx\ds \leq \,c(L,\mathbb Q,\varrho_0,\bfq_0)
\end{equation}
uniformly in $n$ and $l$.
Hence summing over $i$ we get
\begin{equation}\label{47}
    \int_0^{T} \int_{\mt} \tilde\eta_n^{\Gamma+1}\dx\ds \leq \,c(L,\mathbb Q,\varrho_0,\bfq_0).
\end{equation}
Note that this bound is independent of $n$ and $l$,
 but strongly depends on $L$ and $\delta$ as well as on the random element $\tilde\omega$.

%
%
%
%

\subsection{Strong convergence of the density}
\label{subsec:strong}

Throughout this section, without specifying it explicitly, we continue working with the $\tilde\omega$ fixed in Section~\ref{sec:pressure} and we also make use of the time points $t_{i}(\tilde\omega)$, $i=1,\dots,M(\tilde\omega)$ defined there.

First observe that the uniform bound \eqref{47} yields
\begin{align}\label{eq:ll}
	p_\delta(\tilde\eta_n(\tilde\omega))\rightharpoonup \overline{p_\delta(\tilde\eta(\tilde\omega))}\quad\text{in}\quad L^1((0,T)\times\mt).
\end{align}
In particular, the limit is an $L^{1}$-function, not only a Radon measure. This is a direct consequence of  the fundamental theorem of Young measures  similar to Proposition~\ref{skorokhod} but applied only to $\tilde\omega$ and based on the improved integrability provided by the pressure estimate \eqref{47}. More precisely, this approach yields a subsequence (possibly depending on $\tilde\omega$) and the convergence to some limit
characterised by the corresponding Young measure. However, this Young measure must coincide with the evaluation $\tilde\nu(\tilde\omega)$ of the limit Young measure obtained in Proposition~\ref{skorokhod} and hence is independent of the subsequence. In other words,  we obtain \eqref{eq:ll} without passing to an $\tilde\omega$-dependent subsequence.

In order to conclude  the proof of Theorem~\ref{thm:mainl} we shall show that 
$\overline{p_\delta(\tilde\eta)}=p_\delta(\tilde\eta)$ 
in \eqref{eq:momdWB}.
The key is the analysis of the effective viscous flux, which is in our setting formally given by
\begin{equation*}
(\lambda + 2 \mu) \Div^{\tilde\bfphi}\tilde\bfv-p_\delta(\tilde\eta)=\tilde G
\end{equation*}
with
\begin{align} 
\label{G1} \tilde G:&= (-\Delta^{\tilde\bfphi})^{-1}  \div^{\tilde\bfphi}\left( \partial_t \big((\tfrac{1}{n}+[\tilde\eta]_l ) \tilde\bfv\big) +
    \div([\bfA^\top_{\tilde\bfphi}\tilde\eta \tilde\bfv]_l \otimes \tilde\bfv ) 
    -\tfrac{1}{2n} [\nabla \tilde\eta]_l \nabla \tilde\bfv \right)\\
\label{G2}&= 
    (-\Delta^{\tilde\bfphi})^{-1}  \div^{\tilde\bfphi} \Big(  \big(\tfrac{1}{n}+[\tilde\eta]_l \big)  \partial_t \tilde\bfv +
    [\bfA_{\tilde\bfphi}^\top\tilde\eta \tilde\bfv]_l \cdot \nabla\tilde \bfv +\tfrac{1}{n} [\Delta \tilde\eta]_l \tilde\bfv 
    -\tfrac{1}{2n} [\nabla \tilde\eta]_l \nabla \tilde\bfv \Big),
\end{align}
where the equality in the second line follows form the equation of continuity.

We aim to prove
\begin{align}\label{eq:fluxpsi}
\begin{aligned}
\int_0^{\tau} \int_{\mt} &\big(p_\delta(\tilde\eta_n)-(\lambda+2\mu)\Div^{\tilde\bfphi_n} \tilde{\bfv}_n\big)\,\tilde\eta_n\dxt\\&\longrightarrow \int_0^{\tau}\int_{\mt} \big( \overline{p_\delta(\tilde\eta)}-(\lambda+2\mu)\Div^{\tilde\bfphi} \tilde{\bfv}\big)\,\tilde\eta\dxt
\end{aligned}
\end{align}
as $n\rightarrow\infty$ for any fixed $0 < \tau \leq T$. The proof is based on rearranging the terms in \eqref{eq:2711} and passing to the limit in
the corresponding terms on the right-hand side. Without loss of generality, we may suppose that $t_1 = 0$, $t_M = \tau$.

First, we shall show for $i\in\{0,\dots,M-1\}$ that
\begin{align}\label{eq:2711a}
\lim_{n\rightarrow\infty}\int_{\mt}\big(\tfrac{1}{n}+ [\tilde\eta_n]_l\big) \tilde\bfv_n\cdot \tilde\Phi_{i,n}\dx\bigg|_{s=t_i}^{s=t_{i+1}}&=\int_{\mt}[\tilde\eta]_l \tilde\bfv\cdot \tilde\Phi_{i}\dx\bigg|_{s=t_i}^{s=t_{i+1}}\\\lim_{n\rightarrow\infty}\int_{t_i}^{t_{i+1}}\int_{\mt}\Big(\big(\tfrac{1}{n}+ [\Tilde\eta_n]_l\big) \tilde\bfv_n\cdot\partial_t\tilde\Phi_{i,n}&+[\bfA^\top_{\tilde\bfphi_n}\tilde\eta_n \tilde\bfv_n]_l\otimes \tilde\bfv_n:\nabla\tilde\Phi_{i,n}\Big)\dxt\nonumber\\
=\int_{t_i}^{t_{i+1}}\int_{\mt}\Big([\tilde\eta]_l \tilde\bfv\cdot\partial_t\tilde\Phi_{i}&+[\bfA^\top_{\tilde\bfphi}\tilde\eta \tilde\bfv]_l\otimes \tilde\bfv:\nabla\tilde\Phi_{i}\Big)\dxt,\label{eq:2711b}\\
\lim_{n\rightarrow\infty}\tfrac{1}{2n}\int_{t_i}^{t_{i+1}}\int_{\mt}[\nabla\tilde\eta_n]_l\nabla\tilde\bfv_n\cdot\tilde\Phi_{i,n} \dxs&=0,\label{eq:2711d}
\end{align} 
where $\tilde\Phi_i=\nabla^{\tilde\bfphi(t_i)}(-\Delta^{\tilde\bfphi(t_i)})^{-1}(\tilde\eta-m_0)$.
The convergences  \eqref{eq:2711a} and  \eqref{eq:2711d} follow directly from Proposition \ref{skorokhod}.
To see \eqref{eq:2711b}, we compute
\begin{align}
\int_{t_i}^{t_{i+1}}&\int_{\mt}\Big([\tilde\eta_n]_l \tilde\bfv_n\cdot\partial_t\tilde\Phi_{i,n}+[\bfA^\top_{\tilde\bfphi_n}\tilde\eta_n \tilde\bfv_n]_l\otimes \tilde\bfv_n:\nabla\tilde\Phi_{i,n}\Big)\dxs\nonumber\\
&=-\int_{t_i}^{t_{i+1}}\int_{\mt}[\tilde\eta_n]_l \tilde\bfv_n\cdot\nabla^{\tilde\bfphi_n(t_i)}(-\Delta^{\tilde\bfphi_n(t_i)})^{-1}\Div^{\tilde\bfphi_n}(\tilde\eta_n\tilde\bfv_n)\dxs\nonumber\\
&+\tfrac{1}{n}\int_{t_i}^{t_{i+1}}\int_{\mt}[\tilde\eta_n]_l \tilde\bfv_n\cdot\nabla^{\tilde\bfphi_n(t_i)}(-\Delta^{\tilde\bfphi_n(t_i)})^{-1}\Delta\tilde\eta_n\dxs\nonumber\\
&+\int_{t_i}^{t_{i+1}}\int_{\mt}[\bfA^\top_{\tilde\bfphi_n}\tilde\eta_n \tilde\bfv_n]_l\otimes \tilde\bfv_n:\nabla\nabla^{\tilde\bfphi_n(t_i)}(-\Delta^{\tilde\bfphi_n(t_i)})^{-1}\tilde\eta_n\dxs\label{eq:comm}
\end{align}
The last two terms are arbitrarily small by \eqref{times}, while the regularising effect of $[\cdot]_l$ for fixed $l$ allows to pass to the limit in the first two terms as a consequence of  Proposition \ref{skorokhod} (Lemma \ref{lem:divcurl} is not needed at this stage). 

Noticing that
\begin{align*}
&\int_{t_i}^{t_{i+1}}\int_{\mt}p_\delta(\tilde\eta_n)\big(\div^{\tilde\bfphi_n(t_i)}-\div^{\tilde\bfphi_n}\big) \tilde\Phi_{i,n} \dxs\\\text{and}\,\, &\int_{t_i}^{t_{i+1}}\int_{\mt}\mathbb S(\nabla^{\tilde\bfphi_n} \tilde\bfv_n):\big(\nabla^{\tilde\bfphi_n}-\nabla^{\tilde\bfphi_n(t_i)}\big) \tilde\Phi_{i,n} \dxs,
\end{align*}
as well as
\begin{align*}
&\int_{t_i}^{t_{i+1}}\int_{\mt}p_\delta(\tilde\eta)\big(\div^{\tilde\bfphi(t_i)}-\div^{\tilde\bfphi}\big) \tilde\Phi_{i} \dxs\\\,\,\text{and}\,\,&\int_{t_i}^{t_{i+1}}\int_{\mt}\mathbb S(\nabla^{\tilde\bfphi} \tilde\bfv):\big(\nabla^{\tilde\bfphi}-\nabla^{\tilde\bfphi(t_i)}\big) \tilde\Phi_{i} \dxs
\end{align*}
are arbitrarily small by \eqref{times} (choosing $\varkappa$ accordingly) we obtain \eqref{eq:fluxpsi} after summation with respect to $i$.

In order to proceed we need to derive the renormalized equation of continuity. 
For that purpose it is useful to return to the original variables $\tilde\varrho:=\tilde\eta\circ\tilde\bfphi^{{-1}}$ and
$\tilde\bfu:=\tilde\bfv\circ\tilde\bfphi^{{-1}}$ satisfying
\begin{align*}
\dd\tilde\varrho+\Div(\tilde\varrho\tilde\bfu)\dt&=\Div\big(\tilde\varrho\mathbb Q\big)\circ \dd \tilde W
\end{align*} 
 in the sense of distributions.
Arguing as in \cite{BFHZ} we can derive the renormalized version of the continuity equation
\begin{align*}
\dd\theta(\tilde\varrho)+\Div(\theta(\tilde\varrho)\tilde\bfu)\dt&=(\theta(\tilde\varrho)-\theta'(\tilde\varrho)\tilde\varrho)\div\tilde\bfu\dt+\Div\big(\theta(\tilde\varrho)\mathbb Q\big)\circ \dd \tilde W
\end{align*} 
for any $\theta\in BC^2([0,\infty))$ in the sense of distributions.
Transforming back we obtain
\begin{align}\label{eq:ren:delta}
\partial_t\theta(\tilde\eta)+\Div^{\tilde\bfphi}(\theta(\tilde\eta)\tilde\bfv)&=\big(\theta(\tilde\eta)-\theta'(\tilde\eta)\tilde\eta\big)\div^{\tilde\bfphi}\tilde\bfv
\end{align} 
for any $\theta\in BC^2([0,\infty))$ in the sense of distributions.

By the monotonicity of the mapping $\varrho\mapsto p_\delta(\varrho)$, we deduce 
from \eqref{eq:fluxpsi} and the convergences stated in Proposition \ref{skorokhod} that
\begin{align*}
(\lambda+2\mu)\liminf_{n\rightarrow\infty}& \int_0^{\tau} \int_{\mt} \big(\Div^{\tilde\bfphi_n} \tilde\bfv_n\,\tilde\eta_n -\Div^{\tilde\bfphi} \tilde\bfv\,\Tilde\eta\big)\dxt\\
=&\liminf_{n\rightarrow\infty}\int_0^{\tau} \int_{\mt} \big(p_\delta(\tilde\eta_n)- \overline{p_\delta(\tilde\eta)}\big)\big(\tilde\eta_n-\tilde\eta\big)\dxt\geq 0
\end{align*}
for any $0 < \tau \leq T$.
We conclude that
\begin{equation} \label{8.12}
\liminf_{n\rightarrow\infty} \int_0^{\tau} \int_{\mt} \Div^{\tilde\bfphi_n} \tilde\bfv_n\,\tilde\eta_n \dxt	 \geq \int_0^{\tau} \int_{\mt} \Div^{\tilde\bfphi} \tilde\bfv\,\Tilde\eta \dxt,\ 0 < \tau \leq T.	
\end{equation}	

Now, we compute both integrals in
\eqref{8.12} by means of the corresponding continuity equations. Since $(\tilde\bfv_n,\tilde\eta_n)$ is a strong solution to
\eqref{sys-apL}$_1$ (which can be shown by parabolic maximum regularity theory) it is also a renormalized solution and we have
\begin{align*}
\partial_t\theta(\tilde\eta_n)+\Div^{\tilde\bfphi_n}(\theta(\tilde\eta_n)\tilde\bfv_n)&=(\theta(\tilde\eta_n)-\theta'(\tilde\eta_n)\tilde\eta_n)\Div^{\tilde\bfphi_n}\tilde\bfv_n+\tfrac{1}{n}\Delta\theta(\tilde\eta_n)-\tfrac{1}{n}\theta''(\tilde\eta_n)|\nabla\tilde\eta_n|^2
\end{align*} 
for any $\theta\in BC^2([0,\infty))$.
Choosing $\theta(z)=z\ln z$ and integrating in space-time we gain
\begin{align}\label{8.15}
\int_0^{\tau} \int_{\mt}\Div^{\tilde\bfphi_n} \tilde\bfv_{n}\,\tilde\eta_{n}\dxs \leq\int_{\mt}\varrho_0\ln(\varrho_0)\dx
-\int_{\mt}\tilde\eta_n(\tau)\ln(\tilde\eta_n(\tau)\dx
\end{align}
for any $0 < \tau \leq T$.
Similarly, equation \eqref{eq:ren:delta} yields 
\begin{align}\label{8.14}
\int_0^{\tau}  \int_{\mt}\Div^{\tilde\bfphi} \tilde\bfv\,\tilde\eta\dxs=\int_{\mt}\varrho_0\ln(\varrho_0)\dx
-\int_{\mt}\tilde\eta(\tau)\ln(\Tilde\eta(\tau))\dx.
\end{align}
Combining \eqref{8.12}--\eqref{8.14} we may infer that
\begin{align*}
\limsup_{n\rightarrow\infty}\int_{\mt}\tilde\eta_n(t)\ln(\tilde\eta_n(t))\dx\leq \int_{\mt}\eta(t)\ln(\eta(t))\dx
\end{align*}
for any $t\in I$.
This gives the desired strong convergence $\tilde\eta_{n}\rightarrow\tilde\eta$ in $L^1(I\times\mt)$ by strict convexity of $z\mapsto z\ln z$. Consequently, we conclude $\overline{p_{\delta}(\tilde\eta)}=p_\delta(\tilde\eta)$ and the proof of Theorem \ref{thm:mainl} is complete.

\section{Proof of Theorem \ref{thm:main}}
\label{s:5}

The goal of this section is to remove the additional regularisation parameters $l$ and $\delta$ and thus conclude the proof of Theorem \ref{thm:main}.

\subsection{The limit $l\rightarrow0$}
The aim is to pass to the limit $l\rightarrow0$ in \eqref{i1Bl} obtaining the system
\begin{align} \label{i1Bdelta}
\partial_t \eta + \Div^{\bfphi} (\eta\bfv)  &=0,\\
\label{i2Bdelta}
\partial_t (\eta\bfv) + \Div \big(\bfA^\top_{\bfphi} \eta\bfv\otimes\bfv\big) -\Div^{\bfphi}\mathbb S(\nabla^{\bfphi} \bfv) + \nabla^{\bfphi} p_\delta(\eta)  &=0.
\end{align}
The solution to \eqref{i1Bdelta}--\eqref{i2Bdelta} is understood exactly as in Definition \ref{MD1B}
replacing $p$ by $p_\delta$. We aim to prove the following result.

\begin{Theorem}\label{thm:maindelta}
Suppose that $(\varrho_0,\bfq_0)\in L^\Gamma(\mt)\times L^{\frac{2\Gamma}{\Gamma+1}}(\mt)$ with $\varrho_0\geq0$ a.a.  and $\varrho_0^{-1/2}\bfq_0\in L^2(\mt)$ 
are given deterministic initial data. 
 Then there is a dissipative martingale solution to \eqref{i1Bdelta}--\eqref{i2Bdelta} with the initial condition $(\varrho_{0},\bfq_{0})$. In addition, the continuity equation holds in the renormalised sense, that is we have
\begin{align*}
\begin{aligned}
\int_{\mt}\theta(\eta)\psi\dx\bigg|_{s=0}^{s=t}&-
\int_0^t\int_{\mt}\theta(\eta) \bfv\cdot\nabla^{\bfphi}\psi\dxs=\int_0^t\int_{\mt}\big(\theta(\eta)-\theta'(\eta)\eta\big)\div^{\bfphi}\bfv\,\psi\dxs
\end{aligned}
\end{align*}
$\mathbb P$-a.s. for any $\theta\in BC^2([0,\infty))$.
\end{Theorem}
\begin{proof}
The beginning of the proof is rather similar to the proof of Theorem \ref{thm:mainl} performed in the previous section. Thus we only sketch the main ideas. By Theorem  \ref{thm:mainl} we obtain a sequence of martingale solutions $$\big((\Omega,\mf,(\mf^l_t)_{t\geq0},\prst),\eta_l,\bfv_l,W_{l})$$
to  \eqref{i1Bl}--\eqref{i2Bl}. Note that one can assume that the probability space $(\Omega,\mf,\prst)$ is independent of $l$.
\begin{itemize}
\item First of all we can use the energy inequality \eqref{eq:enedW}
to deduce
\begin{align*}
\sup_{I}\|\eta_l\|_{L^\Gamma(\mt)}^\Gamma
+\sup_{I}\big\|\sqrt{[\eta_l]_l}\bfv_l\big\|^2_{L^2(\mt)}+\int_I\|\nabla^{\bfphi_l}\bfv_l\|^2_{L^2(\mt)}\dt&\leq\,C(\varrho_0,\bfq_0)
\end{align*}
uniformly in the sample space and uniformly in $l$. 
\item As in Proposition \ref{skorokhod} we obtain the tightness of the  corresponding laws, where the path space did not change. In particular, there is a nullsequence $(l_n)$ and a sequence of random variables with the same law such that
\begin{equation}\label{49a}
        \begin{cases}
        \Tilde{\eta}_{l_n} \to \Tilde{\eta} \,\, \text{in }\,\,  C_{\rm{weak}}\big([0,T]; L^\Gamma(\mt)\big),\\
   [\Tilde{\eta}_{l_n}]_{l_n}  \Tilde{\bfv}_{l_n} \to \Tilde{\eta}\Tilde{\bfv} \,\, \text{in }\,\,  C_{\rm{weak}}\big([0,T]; L^{2\frac{\Gamma}{\Gamma+1}}(\mt)\big),\\
   \Tilde{\bfv}_{l_n} \rightharpoonup \Tilde{\bfv} \,\, \text{in }\,\,  L^2\big(0,T;W^{1,2}(\mt)),\\
     p_\delta( \tilde \eta_{l_n} ) \to  \overline{p_\delta(\tilde \eta)} \ \mbox{in}\ 
   \left( L^\infty(0,T; \mathcal{M}(\TN)), w^* \right) \\
         \tilde\bfphi_{l_{n}}\to\tilde\bfphi\,\, \text{in} \,\,  C^{\alpha}([0,T];C^{2}(\TN,\TN),\\
         \Tilde{W}_{l_{n}} \to \Tilde{W} \,\, \text{in} \,\,  C([0,T];R^{K}),\\
         \tilde\nu_{l_{n}}\to\tilde\nu\,\, \text{in} \,\, (L^{\infty}((0,T)\times\TN;{\rm Prob}(R^{1+N+N^{2}})),w^{*})
        \end{cases}
    \end{equation}
    $\Tilde{\p}$-a.s. as well as
    \[
    H( \Tilde\eta_{l_{n}} ,\tilde{\bfv}_{l_n}, \nabla\tilde{\bfv}_{l_n}) \rightharpoonup \overline{H(   \tilde\eta,\tilde{\bfv}, \nabla\tilde{\bfv})}\quad\text{in}\,\,L^{k}((0,T)\times\TN)
    \]
    holds $\tilde{\p}$-a.s. as  $n \to \infty$ for all $1 < k\leq\frac{\Gamma}{q_1}\wedge \frac{2}{q_2}$. 
    Here, we have used the fact that the limits of $\tilde \eta_{l_n}$ and $[\tilde \eta_{l_n}]_{l_n}$ in $C_{\rm{weak}}\big([0,T]; L^\Gamma(\mt)\big)$ 
    coincide.

    Finally, define again filtrations $(\tilde\mf^{l_n}_t)_{t\geq0}$ and $(\tilde\mf_t)_{t\geq0}$ to ensure the correct measurability.
\item Also the pressure estimate can be performed via the same strategy as in the previous section obtaining
\begin{equation}\label{49'}
    \int_0^{T} \int_{\mt} \tilde\eta_{l_n}^{\Gamma+1} \dx\ds \leq \,c(L,\mathbb Q,\varrho_0,\bfq_0),
\end{equation}
where the right hand side depends on $\omega$.
\end{itemize}

\medskip 
Now we must prove that the limit object
$$\big((\tilde\Omega,\tilde\mf,(\tilde\mf_t)_{t\geq0},\tilde\prst),\tilde\eta,\tilde\bfv,\tilde W)$$
solves the momentum equation (as for the continuity equation and the energy inequality this follows from the convergences above). The key is again  analysis of the effective viscous flux identity,
which will performed in the remaining part of this subsection. The limit in the momentum equation follows similarly to the last section using standard properties of the convolution.

As in the last section we argue pathwise, i.e. the following computations (and constants) depend on $\tilde\omega\in\tilde\Omega$ which is chosen such that the convergences in \eqref{49a} hold.  
 The effective viscous flux at this stage is given by
\begin{align}\label{G1'}
\begin{aligned}
\tilde G_{l_n}&:=  (\lambda+2\mu)\div^{\tilde \bfphi_{l_n}} \tilde\bfv_{l_n} - p_\delta(\tilde\eta_{l_n}) \\&= (-\Delta^{\tilde\bfphi_{l_n}})^{-1} \div^{\tilde\bfphi_{l_n}} \left( \partial_t ([\tilde\eta]_{l_n} \tilde\bfv_{l_n}) + \div 
([\bfA^T_{\tilde\bfphi_{l_n}} \tilde\eta_{l_n} \tilde\bfv_{l_n}]_{l_n} \otimes \tilde\bfv_{l_n}) \right),\\
\tilde G&:=  (\lambda+2\mu)\div^{\tilde \bfphi} \tilde\bfv - \overline{p_{\delta}(\tilde\eta)} \\&= (-\Delta^{\tilde\bfphi})^{-1} \div^{\tilde\bfphi} \left( \partial_t (\tilde\eta \tilde\bfv) + \div 
([\bfA^T_{\tilde\bfphi} \tilde\eta \tilde\bfv \otimes \tilde\bfv) \right).
\end{aligned}
\end{align}
Alternatively, this can be expressed as
\begin{align}\label{G2'}
\begin{aligned}
  \tilde G_{l_n}&= (-\Delta^{\tilde\bfphi_{l_n}}_{l_n})^{-1} \div^{\tilde\bfphi_{l_n}} \left( [\tilde\eta_{l_n}]_{l_n} \partial_t \tilde\bfv_{l_n} + [\bfA^T_{\tilde\bfphi_{l_n}} \tilde\eta_{l_n} \tilde\bfv_{l_n}]_{l_n} \nabla \tilde\bfv_{l_n} \right),\\
  \tilde G&= (-\Delta^{\tilde\bfphi})^{-1} \div^{\tilde\bfphi} \left( \tilde\eta \partial_t \tilde\bfv + \bfA^\top_{\tilde\bfphi} \tilde\eta \tilde\bfv \nabla \tilde\bfv \right).
\end{aligned}
\end{align}
Here the terms in brackets can be interpreted in negative Sobolev spaces as a consequence of the corresponding momentum equation. By maximal regularity theory for elliptic equations this is sufficient to verify that the expressions for $\tilde G_{l_n}$ and  $\tilde G_{l_n}$ above are indeed measurable functions.
%
%
{Above we shall mention about the definition of $\tilde\eta \partial_t \tilde\bfv$.
It shall be treated as a distribution meant in the sense $\tilde\eta \partial_t \tilde\bfv= 
\partial_t( \tilde\eta  \tilde\bfv) - (\partial_t \tilde\eta)  \tilde\bfv$. In case of low integrability of the density the last definition is just formal, and it is shall be meant as in the definition below Lemma \ref{ChMZ-lem}, which bases on the regularisation of $\tilde \bfv$.}

Now, we introduce the cut-off function
\begin{align}\label{eq:Tk'}
	T_k(z):=k\,T\Big(\frac{z}{k}\Big),\quad z\in\R,\,\, k\in\mathbb N,
\end{align}
where $T$ is a smooth concave function on $\R$ such that $T(z)=z$ for $z\leq 1$ and $T(z)=2$ for $z\geq3$.
We consider the approximate continuity equation in the renormalised sense, cf. \eqref{eq:condrenlWB}, with 
\begin{align}\label{eq:Lk}
\theta(z)=L_k(z) = z \int_1^z \frac{T_k(r)}{r^2} \ {\rm d}r 
\end{align}
for some $k\in\N$.
Passing to the limit we obtain
\begin{equation}\label{aa'}
  \partial_t(\overline{L_k(\tilde\eta)}) + \div^{\tilde\bfphi} \big( \overline{L_k(\tilde\eta)} \tilde\bfv \big) +\overline{T_k(\tilde\eta) p_\delta(\tilde\eta)} = \overline{T_k(\tilde\eta) \tilde G}
\end{equation}
in the sense of distributions.

We perform the same steps for the limit equation and utilize the properties of renormalised solutions. The renormalised formulation can be derived as in \eqref{eq:ren:delta}, specifically going back to the original variables and applying the regularization procedure 
of DiPerna and Lions. This allows us to conclude
\begin{equation}\label{aa''}
  \partial_t (L_k(\tilde\eta)) + \div^{\tilde\bfphi} (L_k (\tilde\eta)\tilde\bfv) + \overline{T_k(\tilde\eta)}\, \overline{p_\delta(\tilde\eta)} = (\overline{T_k(\tilde\eta)} - T_k(\tilde\eta)) \div^{\tilde\bfphi} \tilde\bfv + \overline{T_k (\tilde\eta)} \tilde G.
\end{equation}

As we shall see in \eqref{58} below, comparing these two identities yields the desired strong convergence of \( \eta \).  At this stage, \( \eta \in L^{\Gamma+1}_{t,x} \) by \eqref{49'}, and thus we observe
\begin{equation}\label{l24}
  \int_0^\tau  \int_{\mt} (\overline{T_k(\tilde\eta)} - T_k ( \tilde\eta) )\div^{\tilde\bfphi} \tilde\bfv \dx \dt \to 0 \quad \text{as} \quad k \to \infty,\ 0 < \tau \leq T.
\end{equation}

The crucial step is to show
\begin{equation}\label{l25}
    \overline{T_k(\tilde\eta)\tilde G}-\overline{T_k(\tilde\eta)} \tilde G \to 0 \quad \text{as} \quad k \to \infty.
\end{equation}


The key problem is the low regularity of the operator $(-\Delta^{\tilde\bfphi_{l_n}})^{-1}  \div^{\tilde\bfphi_{l_n}}$
with respect to the time variable. Fortunately, thanks to (\ref{eq:3010}), we may approximate ${\tilde\bfphi}_{l_n}(t)$ by its mollification $\underline{\tilde\bfphi}_{l_n}(t)$ with respect to the time variable.
 We then replace
\begin{equation*}
     (-\Delta^{\tilde\bfphi_{l_n}})^{-1}  \div^{\tilde\bfphi_{l_n}} \quad \text{with} \quad (-\Delta^{\underline{\tilde\bfphi}_{l_n}})^{-1}  \div^{\underline{\tilde\bfphi}_{l_n}} 
\end{equation*}
so that, given $\sigma' > 0$, the matrices $\bfA_{\tilde{\bfphi}_{l_n}}$ and $\bfA_{\underline{\tilde\bfphi}_{l_n}}$ defined by (\ref{A-phi}) satisfy the estimate
\begin{equation*}
    \|\bfA_{\tilde{\bfphi}_{l_n}}-\bfA_{\underline{\tilde\bfphi}_{l_n}}\|_{L^\infty((0,T)\times\mt)} \leq \sigma',
\end{equation*}
recall \eqref{eq:3010}.
The matrix $\bfA_{\tilde{\bfphi}_{l_n}}$ is H\"older continuous in time, while
$\bfA_{\underline{\tilde\bfphi}_{l_n}}$ is smoothed out by convolution with a parameter proportional to a power of $\sigma'$. 

Now, introducing $\underline{\tilde G}_{l_n}$ as $\tilde G_{l_n}$
with $\underline{\tilde\bfphi}_{l_n}$ instead of ${\tilde\bfphi}_{l_n}$ in (\ref{G1'}) and (\ref{G2'}), we obtain uniformly in $n$
\begin{equation*}
    \|\tilde G_{l_n}-\underline{\tilde G}_{l_n}\|_{L^1(\mt \times (0,T))}\leq C\sigma'
\end{equation*}
as a consequence of \eqref{49a} and elliptic estimates as in \eqref{eq:2311}.
Hence we obtain the $k$-dependent estimate
\begin{equation*}
    \int_0^T \int_{\mt} |T_k (\tilde\eta_{l_n}) (\tilde G_{l_n}-\underline{\tilde G}_{l_n})| \dx\dt \leq 
    C k \sigma'.
\end{equation*}
In general, taking $\sigma' \approx k^{-2}$ ensures that the right-hand side is of order $k^{-1}$. Thus
\eqref{l25} will follow from
\begin{equation*}
    \overline{T_k(\tilde\eta)\underline{\tilde G}}-\overline{T_k(\tilde\eta)} \underline{\tilde G} \to 0 \quad \text{as} \quad k \to \infty.
\end{equation*}

Hence, we need to proceed with the main parts of $\underline{\tilde G}_{l_n}$.
The first step is to show that
\begin{equation}\label{v-sigma}
    \|[\tilde\eta_{l_n} ]_{l_n}  (\tilde\bfv_{l_n} 
   - (\tilde\bfv_{l_n} )_\sigma)\|_{L^1((0,T)\times {\mt} )} \leq \varkappa(\sigma),
\end{equation}
where $(v_{l_n} )_\sigma$ is a regularization of $\tilde\bfv_{l_n} $ by convolution in time and space with a family of regularizing kernels, and
$\varkappa(\sigma)\rightarrow0$ as $\sigma\rightarrow0$. We remind the reader that $\sigma$ is different from $\sigma'$ used for the regularisation of $\tilde\bfphi$ above. 
In \eqref{v-sigma} the right-hand side does not depend on $k, \sigma',$ or $n$. The
desired estimate \eqref{v-sigma} follows from the following result which is proved in \cite[Lemma 12]{ChMZ} combined with \eqref{49a}and the balance of momentum and continuity, respectively.

\begin{Lemma}\label{ChMZ-lem}
	Let $\pi_\sigma$ denote the time and space mollifier with parameter $\sigma>0$. 
	Let $\eta \in L^\infty_tL^\gamma_x$, $\bfv\in L^2_t W^{1,2}_x$, 
	$\partial_t (\eta \bfv ) \in L^{q}_tW^{-1,q}_x$ and $\partial_t \eta   \in L^{q}_tW^{-1,q}_x$  with $ q>1 $.
	Then there exists a positive constant $\theta$ depending on $ \gamma $ and $ q $, such that
	\begin{equation}
		\|\eta |\bfv-\pi_\sigma\ast \bfv|\|_{L^1((0,T)\times\mt)}\leq C \sigma^\theta.
	\end{equation}
The constant $C$ depends on the norms of the quantities in the function spaces mentioned before.
\end{Lemma}
\smallskip  

Thanks to (\ref{v-sigma}) we are able to replace $\tilde\bfv_{l_n}$ by $(\tilde\bfv_{l_n})_\sigma$.
In order to do so we decompose $\underline{\tilde G}_{l_n}$ into the sum of
 \begin{align*}
  \underline{\tilde G}^0_{l_n}&:=  (-\Delta^{\underline {\tilde\bfphi}_{l_n} })^{-1}  \div^{\underline {\tilde \bfphi}_{l_n}} \left(  [\tilde\eta_{l_n} ]_{l_n}    \partial_t (\tilde\bfv_{l_n} )_\sigma +
    [\bfA^T{\underline {\tilde\bfphi}_{l_n} }\tilde\eta_{l_n}  \tilde\bfv_{l_n} ]_{l_n}  \cdot \nabla (\tilde\bfv_{l_n} )_\sigma \right),\\
      \underline{\tilde G}^1_{l_n}&:=(-\Delta^{\underline {\tilde\bfphi}_{l_n}})^{-1}  \div^{\underline {\tilde\bfphi}_{l_n} } \partial_t \left( [\tilde\eta_{l_n}]_{l_n} (\tilde\bfv_{l_n}-(\tilde\bfv_{l_n})_\sigma) \right),\\
  \underline{\tilde G}^2_{l_n}&:=(-\Delta^{\underline {\tilde\bfphi}_{l_n}})^{-1}  \div^{\underline {\tilde\bfphi}_{l_n} } \div (    [\bfA^T_{\underline {\tilde\bfphi}_{l_n} }\tilde\eta_{l_n}  \tilde\bfv_{l_n} ]_{l_n}  \otimes(\tilde\bfv_{l_n}-(\tilde\bfv_{l_n})_\sigma)),
\end{align*}
and similarly for $\underline{\tilde G}$.
We observe that $\underline{\tilde G}_{l_n}^0$ has compactness in space (as a consequence of \eqref{49a} and because of the mollification) and $T_k(\tilde \eta_{l_n})$ in time with negative values in space (because of the renormalised continuity equation), so as $n \to 0$ we get
\begin{equation*}
    \overline{T_k(\tilde \eta) \underline{\tilde G}^0}=\overline{T_k(\tilde \eta)} \underline{\tilde G}^0.
\end{equation*}
Then studying $T_k(\tilde\eta_{l_n} ) \underline{\tilde G}^1_{l_n}$ we have 
\begin{align*}
    \int_0^T \int_{\mt}& (-\Delta^{\underline {\tilde\bfphi}_{l_n}  })^{-1}  \div^{\underline {\tilde\bfphi}_{l_n} }  \partial_t \left( [\tilde\eta_{l_n} ]_{l_n}  ) (\tilde\bfv_{l_n} -(\tilde\bfv_{l_n} )_\sigma \right) T_k(\tilde\eta_{l_n} )  \phi \dx\dt\\&=
    \int_0^T \int_{\mt}  \left(  [\tilde\eta_{l_n} ]_{l_n}  (\tilde\bfv_{l_n} -(\tilde\bfv_{l_n} )_\sigma) \right) \cdot\partial_t (-\Delta^{\underline {\tilde\bfphi}_{l_n}  })^{-1}  \nabla^{\underline {\tilde\bfphi}_{l_n}  } (T_k(\tilde\eta_{l_n} ) \phi) \dx\dt
\end{align*}
for any test function $\phi\in \mathcal{D}( 0,T)$.
Due to \eqref{49a} we have (since $\underline {\tilde\bfphi_{l_n} } $ is smooth in time)
\begin{equation*}
    \partial_t (-\Delta_{\underline {\tilde\bfphi}_{l_n}  })^{-1}  \div_{\underline {\tilde\bfphi}_{l_n}  } T_k (\tilde\eta _{l_n}) 
    \in L^2(0,T;L^6(\mt))
\end{equation*}
uniformly in $n$ with a bound depending on $\sigma'$ and (linearly) on $k$, i.e.,
\begin{equation*}
    \|    \partial_t (-\Delta_{\underline {\tilde\bfphi}_{l_n}  })^{-1}  \div_{\underline {\tilde\bfphi}_{l_n}  } T_k (\tilde\eta _{l_n}) 
    \|_{L^2(0,T;L^6(\mt))} \leq C_{\sigma'}k,
\end{equation*}
with a constant $C_{\sigma'}$ that blows up as $\sigma'\rightarrow0$.
 This is a consequence of the renormalised continuity equation \eqref{eq:condrenlWB}, elliptic theory as in \eqref{eq:2311} and Sobolev's embedding $W^{1,2}_x\hookrightarrow L^6_x$. In addition, we have $[\tilde\eta_{l_n}]_{l_n}  (\tilde\bfv_{l_n}-(\tilde\bfv_{l_n})_\sigma) \in L^{\infty}_tL^{\frac{2\Gamma}{\Gamma+1}}_x$ uniformly in $n$ and $\sigma$
by \eqref{49a}. Thus, in combination with \eqref{v-sigma}, we conclude that $[\tilde\eta_{l_n}]_{l_n}  (\tilde\bfv_{l_n}-(\tilde\bfv_{l_n})_\sigma)$
has a small norm in $L^2_tL^{6/5}_x$ in terms of  $\sigma$.
Hence, reducing the value of $\theta$ as the case may be,
\begin{equation}\label{eq:0702}
    \bigg|\int_0^T \int_{\mt} T_k(\tilde\eta_{l_n}) \underline{\tilde G}_{l_n}^1 \phi \dx\dt\bigg| \leq C_{\sigma'}k\sigma^\theta.
\end{equation}
Given $k$, we fix $\sigma' \approx k^{-2}$
and choose $\sigma$ 
so small (depending on $k$) that the right-hand side is of order $k^{-1}$.
Finally, we have by \eqref{49a} (reducing possibly again the value of $\theta$)
\begin{align*}
 \int_0^T\Big\|\underline{\tilde G}^2_{l_n}\Big\|_{L^{\frac{4\Gamma}{3\Gamma+1}}(\mt)}\dt&\leq\,c\int_0^T\Big\|   [\bfA^T_{\underline {\tilde\bfphi}_{l_n} }\tilde\eta_{l_n}  \tilde\bfv_{l_n} ]_{l_n}  \otimes(\tilde\bfv_{l_n}-(\tilde\bfv_{l_n})_\sigma)\Big\|_{L^{\frac{4\Gamma}{3\Gamma+1}}(\mt)}\dt\\
&\leq\,c\int_0^T\|\tilde\eta_{l_n}  \tilde\bfv_{l_n}\|_{L^{\frac{2\Gamma}{\Gamma+1}}}\|\tilde\bfv_{l_n}-(\tilde\bfv_{l_n})_\sigma\|_{L^{\frac{4\Gamma}{\Gamma-1}}(\mt)}\dt\\
&\leq\,c\sigma^\theta\int_0^T\|\tilde\bfv_{l_n}\|_{W^{1,2}(\mt)}\dt\leq\,c\sigma^\theta
\end{align*}
using standard properties of the regularisation.
Note that we used $\frac{4\Gamma}{3\Gamma+1}>1$ and
$\frac{4\Gamma}{\Gamma-1}<6$. This allows us to prove a counterpart of \eqref{eq:0702} for $\underline{\tilde G}_{l_n}^2$.
Hence returning to \eqref{aa'} and \eqref{aa''} we obtain
\begin{equation*}
    \frac{\dd}{\dt}\int_{\mt} 
   \Big( \overline{(\tilde\eta \ln (\tilde\eta))_k } - (\tilde\eta \ln (\tilde\eta))_k \Big)\dx + \int_{\mt}\Big( \overline{p(\eta)\,T_k(\tilde\eta)} - \overline{p(\tilde\eta)}\, \overline{T_k(\tilde\eta)}\Big) \dx \leq \,Ck^{-1}.
\end{equation*}
Thanks to the sufficient integrability of the density, we may perform the limit $k \to \infty$ to obtain $\overline{\tilde\eta \ln(\tilde\eta)} = \tilde\eta \ln (\tilde\eta) $. 
Thus we conclude
\begin{equation}\label{58}
    \tilde\eta^{l_n} \to \tilde\eta\quad \mbox{strongly in \ \ } L^2((0,T)\times\mt) \mbox{ \ \ as \ \ } 
    n\to \infty,
\end{equation} 
which finishes the proof of Theorem \ref{thm:maindelta}.
\end{proof}

\subsection{The vanishing artificial pressure limit}
Inn this final subsection we conclude the proof of Theorem \ref{thm:main}.
Given initial data $(\varrho_0,\bfq_0)\in L^\gamma(\mt)\times L^{\frac{2\gamma}{\gamma+1}}(\mt)$ with $\varrho_0\geq0$ a.a. such that $\varrho_0^{-1/2}\bfq_0\in L^2(\mt)$ we consider their regularization $(\varrho_0^\delta,\bfq_0^\delta)$ such that $\varrho_0^\delta\in L^\Gamma(\mt)$ with $\varrho_0^\delta\geq0$ a.a., 
 $(\varrho_0^\delta)^{-1/2}\bfq_0^\delta\in L^2(\mt)$ and
\begin{align*}
\int_{\mt}\bigg(P_\delta(\varrho_0^\delta)+\frac{1}{2}\frac{|\bfq_0^\delta|^2}{\varrho_0^\delta}\bigg)\dx\rightarrow \int_{\mt}\bigg(P(\varrho_0)+\frac{1}{2}\frac{|\bfq_0|^2}{\varrho_0}\bigg)\dx
\end{align*}
as $\delta\rightarrow0$.
With Theorem \ref{thm:maindelta} at hand we obtain a sequence of martingale solutions $$\big((\Omega,\mf,(\mf^\delta_t)_{t\geq0},\prst),\eta_\delta,\bfv_\delta,W_\delta)$$
to  \eqref{i1Bdelta}--\eqref{i2Bdelta}. The energy
inequality applies and thus we have
\begin{align}\label{eq:endelta}
\sup_{I}\Big(\delta\|\eta_\delta\|_{L^\Gamma(\mt)}^\Gamma+\|\eta_\delta\|_{L^\gamma(\mt)}^\gamma
+\big\|\sqrt{\eta_\delta}\bfv_\delta\big\|^2_{L^2(\mt)}\Big)+\int_I\|\nabla^{\bfphi_\delta}\bfv_\delta\|^2_{L^2(\mt)}\dt&\leq\,C(\varrho_0,\bfq_0)
\end{align}
uniformly in the sample space and uniformly in $\delta$.  
Changing $\Gamma$ to $\gamma$ in the definition of $\mathfrak X_{\eta}$ 
we obtain analogously to Proposition \ref{skorokhod}:
\begin{Proposition} \label{skorokhoddelta}
There exists a nullsequence $(\delta_n)$ and a complete probability space  $(\Tilde{\Omega}, \Tilde{\FF},\Tilde{\p})$ with $\mathfrak{X}$-valued random variables  $[\tilde\eta_{\delta_n},
\tilde\eta_{\delta_n} \tilde\bfv_{\delta_n},\tilde\bfv_{\delta_n},  \tilde\bfphi_{\delta_{n}},\tilde{W}_{{\delta_n}},\tilde\nu_{\delta_n}]$, $n\in\N$, $[\tilde\eta,
\tilde\eta\tilde\bfv,\tilde\bfv, \tilde\bfphi, \tilde{W},\tilde\nu]$
such that
\begin{itemize}
    \item [(a)] For all $n \in \N$ the law of $[\tilde\eta_{\delta_n},
\tilde\eta_{\delta_n} \tilde\bfv_{\delta_n},\tilde\bfv_{\delta_n},  \tilde\bfphi_{\delta_{n}},\tilde{W}_{{\delta_n}},\tilde\nu_{\delta_n}]$ on $\mathfrak{X}$ coincides with the law of $$[ \eta_{\delta_n}
,\eta_{\delta_n} \bfv_{\delta_n},\bfv_{\delta_n}, \bfphi_{\delta_{n}}, {W}_{{\delta_n}},  \delta_{[\eta_{\delta_n}, \bfv_{\delta_n},\nabla\bfv_{\delta_n}]}] ;$$ 
    \item[(b)] The sequence $[\tilde\eta_{\delta_n},
\tilde\eta_{\delta_n} \tilde\bfv_{\delta_n},\tilde\bfv_{\delta_n},  \tilde\bfphi_{\delta_{n}},\tilde{W}_{{\delta_n}},\tilde\nu_{\delta_n}]$
 converges $\Tilde{\p}$-a.s.  to  $[\tilde\eta,
\tilde\eta\tilde\bfv,\tilde\bfv, \tilde\bfphi, \tilde{W},\tilde\nu]$ in the topology of $\mathfrak{X}$;
    \item[(c)] For any Carath\'eodory function $H =H(t,x,\eta,\bfw, \bfW)$ where $(t,x)\in (0,T)\times\TN$, $ (\eta,\bfw,\bfW)\in \R^{1+N+N^2}$, satisfying for some $q_1,q_2>0$ the growth condition
    \[
    H(t,x,\eta,\bfv, \bfW)\lesssim 1 +|\eta|^{q_1}+|\bfw|^{q_2}+|\bfW|^{q_2}
    \]
    uniformly in $(t,x)$, we denote by $\overline{H(\tilde\eta,\tilde{\bfv}, \nabla\tilde{\bfv})}(t,x) =\langle \tilde\nu_{t,x},H \rangle$. Then the following
    \[
    H(   \tilde\eta_{\delta_{n}} ,\tilde{\bfv}_{\delta_n}, \nabla\tilde{\bfv}_{\delta_n}) \rightharpoonup \overline{H(   \Tilde{\eta} ,\tilde{\bfv}, \nabla\tilde{\bfv})}\quad\text{in}\,\,L^{k}((0,T)\times\TN)\quad\text{for all }L\in\N
    \]
    holds $\tilde{\p}$-a.s. as  $n \to \infty$ for all $1 < k\leq\frac{\gamma}{q_1}\wedge \frac{2}{q_2}$.
    
\end{itemize}
\end{Proposition}
We define again filtrations $(\tilde\mf^{\delta_n}_t)_{t\geq0}$ and $(\tilde\mf_t)_{t\geq0}$ as the natural ones to ensure the correct measurability. By Proposition
\ref{skorokhoddelta} we can pass to the limit in the continuity equation and the energy inequality and obtain similarly to \eqref{eq:momdWB}
\begin{align}\label{eq:momdWBdelta}
\begin{aligned}
\int_{\mt} \tilde\eta \tilde\bfv&\cdot \bfpsi\dx\bigg|_{s=0}^{s=t} -\int_0^{t}\int_{\mt}\tilde\eta \tilde\bfv\otimes \tilde\bfv:\nabla^{\tilde\bfphi} \bfpsi\dxs
\\
&+\int_0^{t}\int_{\mt}\mathbb S(\nabla^{\tilde\bfphi} \tilde\bfv):\nabla^{\tilde\bfphi} \bfpsi \dxs-\int_0^{t}\int_{\mt}
\overline{p(\tilde\eta)}\,\Div^{\tilde\bfphi} \bfpsi \dxs
=0
\end{aligned}
\end{align}  
$\Tilde{\mathbb P}$-a.s. for all $\bfpsi\in C^\infty(\mt,R^{N})$. 

Now we prove higher integrability of the pressure which requires a modification of the pathwise method from  Section \ref{sec:pressure}.  We decompose again the time interval into random time intervals $(t_i,t_{i+1})$ such that for all $i\in\{1,\dots,M\}$
\begin{equation}\label{eq:1012delta}
    \|\bfA_{\tilde{\tilde\bfphi}_\delta(t)}-\bfA_{\tilde{\bfphi}_\delta(t_i)}\|_{L^\infty(\mt)} \leq \varkappa \mbox{ \ \ for all \ \ }  t \in (t_i,t_{i+1}).
\end{equation}
For each $i$ we define
\begin{equation*}
\tilde\Phi_{i,\delta}=-\nabla^{\tilde\bfphi_\delta(t_i)}(-\Delta^{\tilde\bfphi_\delta(t_i)})^{-1}\bigg(\tilde\eta^\Theta_\delta-\int_{\mt}\tilde\eta_\delta^\Theta\dx\bigg)
\end{equation*}
where $\Theta\leq \tfrac{2}{N}\gamma-1$. We get by standard elliptic theory
\begin{align}\label{eq:2311delta}
\|\nabla\tilde\Phi_{i,\delta}\|_{L^p(\mt)}&\leq\,c(L)\big(\|\tilde\eta_\delta^\Theta\|_{L^p(\mt)}+1\big)
\end{align}
uniformly in $n$ for $t \in (t_i,t_{i+1})$.

We compute
\begin{align*}
\int_{t_i}^{t_{i+1}}\int_{\mt}p_\delta(\tilde\eta_n)\bigg(\eta^\Theta_\delta-\int_{\mt}\tilde\eta_\delta^\Theta\dx\bigg)\dxs&=\int_{t_i}^{t_{i+1}}\int_{\mt}p_\delta(\tilde\eta_n)\div^{\tilde\bfphi(t_i)} \tilde\Phi_{i,\delta} \dxs\\
&=\int_{t_i}^{t_{i+1}}\int_{\mt}p_\delta(\tilde\eta_n)\div^{\tilde\bfphi} \Phi_{i,\delta} \dxs\\&+\int_{t_i}^{t_{i+1}}\int_{\mt}p_\delta(\tilde\eta_n)\big(\div^{\tilde\bfphi(t_i)}-\div^{\tilde\bfphi}\big) \tilde\Phi_{i,\delta} \dxs,
\end{align*}
where the second term is controlled by
\begin{align*}
\sup_{(t_i,t_{i+1})}&\|\bfA_{\tilde\bfphi_\delta}-\bfA_{\tilde\bfphi_\delta(t_i)}\|_{L^\infty(\mt)}\int_{t_i}^{t_{i+1}}\|p(\tilde\eta_\delta)\|_{L^{\frac{\gamma+\Theta}{\gamma}}(\mt)}\|\nabla \tilde\Phi_{i,\delta} \|_{L^{\frac{\gamma+\Theta}{\Theta}}(\mt)}\ds\\
&+\delta\sup_{(t_i,t_{i+1})}\|\bfA_{\tilde\bfphi_\delta}-\bfA_{\tilde\bfphi_\delta(t_i)}\|_{L^\infty(\mt)}\int_{t_i}^{t_{i+1}}\|\tilde\eta_\delta^\Gamma\|_{L^{\frac{\Gamma+\Theta}{\Gamma}}(\mt)}\|\nabla \tilde\Phi_{i,\delta} \|_{L^{\frac{\Gamma+\Theta}{\Theta}}(\mt)}\ds\\
&\leq\,\varkappa c(L)\int_{t_i}^{t_{i+1}}\Big(\|\tilde\eta_\delta\|_{L^{\gamma+\Theta}(\mt)}^{\gamma+\Theta}+\delta\|\tilde\eta_\delta\|_{L^{\Gamma+\Theta}(\mt)}^{\Gamma+\Theta}+1\Big)\ds
\end{align*}
as a consequence \eqref{eq:1012delta} and \eqref{eq:2311delta}.
For $L>0$ given we will choose $\varkappa$ so small that this term can be absorbed.
For the first term we have
\begin{align*}
\int_{t_i}^{t_{i+1}}\int_{\mt}
p_\delta(\tilde\eta_\delta)\,\Div^{\tilde\bfphi_\delta} \tilde\Phi_{i,\delta} \dxs&=\int_{\mt}\tilde\eta_\delta \tilde\bfv_\delta\cdot \tilde\Phi_{i,\delta}\dx\bigg|_{s=t_i}^{s=t_{i+1}}\\\nonumber&+\int_{t_i}^{t_{i+1}}\int_{\mt}\tilde\eta_\delta \tilde\bfv_\delta\cdot\partial_t\tilde\Phi_{i,\delta}\dxt\\& -\int_{t_i}^{t_{i+1}}\int_{\mt}\tilde\eta_\delta \tilde\bfv_\delta\otimes \tilde\bfv_\delta:\nabla^{\tilde\bfphi_\delta}\tilde\Phi_{i,\delta}\dxs\nonumber
\\
&+\int_{t_i}^{t_{i+1}}\int_{\mt}\mathbb S(\nabla^{\tilde\bfphi_\delta} \tilde\bfv_\delta):\nabla^{\tilde\bfphi_\delta} \tilde\Phi_{i,\delta} \dxs\nonumber\\
&=:(\tt{I})^{i,\delta}+\dots+(\tt{IV})^{i,\delta}\nonumber
\end{align*} 
by the approximate momentum equation.
By \eqref{eq:2311delta} we obtain, decreasing the value of $\Theta$ if needed,
\begin{align*}
(\tt{I})^{i,\delta}
&\leq  \sup_{(t_{i},t_{i+1})}\| \sqrt{\tilde\eta_\delta} \bfv_\delta\|_{L^2(\mt)} \|\sqrt{\tilde\eta_\delta}\|_{L^{2\gamma}(\mt)}\|\tilde\Phi_{i,\delta}\|_{L^{\frac{2\gamma}{\gamma-1}}(\mt)}\\
&\leq c(L)\sup_{(t_{i},t_{i+1})}\| \sqrt{\tilde\eta_\delta} \tilde\bfv_\delta\|_{L^2(\mt)} \|\tilde\eta_\delta\|_{L^{2\gamma}(\mt)}^{\frac{1}{2}}\big(\|\tilde\eta_\delta\|_{L^\gamma(\mt)}+1\big),
\end{align*}
which is uniformly bounded by \eqref{eq:endelta}.

For the term $(\tt II)^{i,\delta}$ the renormalised continuity equation yields
\begin{align*}
\begin{aligned}
   \partial_t \tilde\Phi_{i,\delta}
&=- \nabla^{\tilde\bfphi_\delta(t_i)} (-\Delta^{\tilde\bfphi_\delta(t_i)})^{-1}\div^{\tilde\bfphi_\delta}(\tilde\eta_\delta^\Theta\tilde\bfv_\delta)\\
 &+(1-\Theta)\nabla^{\tilde\bfphi_\delta(t_i)} (-\Delta^{\tilde\bfphi_\delta(t_i)})^{-1}(\tilde\eta^\Theta_\delta\Div^{\tilde\bfphi_\delta}\tilde\bfv_\delta)
\end{aligned}
\end{align*}
in $(t_i,t_{i+1})$. Hence we obtain by elliptic theory for $q_N:=\frac{N\gamma}{2\gamma-N}$
\begin{align*}
({\tt II})^{i,\delta}&\leq\int_{t_i}^{t_{i+1}}\|\tilde\eta_\delta\|_{L^\gamma(\mt)}\|\tilde\bfv_\delta\|^2_{L^{\frac{2N}{N-2}}(\mt)}\|\tilde\eta_\delta^\Theta\|_{L^{q_N}(\mt)}\ds\\
&\leq\,c(L)\sup_{(t_i,t_{i+1})}\|\tilde\eta_\delta\|_{L^\gamma(\mt)}^2\int_{t_i}^{t_{i+1}}\|\nabla^{\tilde\bfphi_\delta}\tilde\bfv_\delta\|^2_{L^2(\mt)}\ds\leq\,c(L,\mathbb Q,\varrho_0,\bfq_0)
\end{align*}
using \eqref{eq:2311delta}.

We can estimate $({\tt III})^{i,\delta}$ along the same lines.
The estimates for $({\tt IV})^{i,\delta}$ are a direct consequence of \eqref{eq:endelta}, \eqref{eq:2311delta} and the choice $2\Theta\leq \gamma$.

Combining everything and summing with respect to $i$ we have shown
\begin{equation*}
    \int_0^{T} \int_{\mt} \big(\delta\tilde\eta_\delta^{\Gamma+\Theta}+\tilde\eta_\delta^{\gamma+\Theta}\big)\dx\ds \leq \,c(L,\mathbb Q,\varrho_0,\bfq_0)
\end{equation*}
uniformly in $\delta$, but with a right hand side that depends on $\tilde\omega$.

It remains to show strong convergence of $\tilde \eta_{\delta}$. 
We recall the definition of $T_k$ from \eqref{eq:Tk'}.
 By Proposition \ref{skorokhoddelta} we clearly have $\tilde{\mathbb P}$-a.s.
\begin{align*}
 T_k(\tilde\eta_{\delta_n})&\rightarrow \tilde{T}^{1,k}\quad\text{in}\quad C_{\rm weak}(I;L^p( \mt))\quad\forall p\in[1,\infty),\\
\big(T_k'(\tilde\eta_{\delta_n})\tilde\eta_{\delta}-T_k(\tilde\eta_{\delta})\big)\Div^{\tilde\bfphi_{\delta_n}} \tilde\bfv_{\delta_n}&\rightharpoonup\tilde{T}^{2,k}
\quad\text{in}\quad L^2(I\times\mt),
\end{align*}
for some limit functions $\tilde{T}^{1,k}$ and $\Tilde{T}^{2,k}$.
Now we have to show that 
\begin{align}\label{eq:flux'}
\begin{aligned}
\int_0^{\tau}\int_{\mt}&\big( p_\delta(\tilde\eta_{\delta_n})-(\lambda+2\mu)\Div^{\tilde\bfphi_{\delta_n}} \tilde\bfv_{\delta_n}\big)\,T_k(\tilde\eta_{\delta_n})\dxt\\&\longrightarrow\int_0^{\tau}\int_{\mt} \big( \overline{p(\tilde\eta)}-(\lambda+2\mu)\Div^{\tilde\bfphi} \tilde\bfv\big)\,\tilde T^{1,k}\dxt
\end{aligned}
\end{align}
for all $0<\tau\leq T.$
In order to prove \eqref{eq:flux'} we test the approximate momentum equation in $(t_i,t_{i+1})$ by $\tilde\Phi_{i,\delta_n}:=(-\nabla^{\tilde\bfphi_{\delta_n}(t_i)}\Delta^{\tilde\bfphi_{\delta_n}(t_i)})^{-1}T_k(\tilde\eta_{\delta_n})$, while \eqref{eq:momdWBdelta} is tested by
$\tilde\Phi_i:=(-\Delta^{\tilde\bfphi(t_i)})^{-1}\nabla^{\tilde\bfphi(t_i)}\tilde T^{1,k}$. We must show for $i=1,\dots,M$ that $\tilde\p$-a.s.
\begin{align*}
\lim_{n\rightarrow\infty}\int_{\mt}\tilde\eta_{\delta_n} \bfv_{\delta_n}\cdot \tilde\Phi_{i,{\delta_n}}\dx\bigg|_{s=t_i}^{s=t_{i+1}}
&=\int_{\mt}\tilde\eta \tilde\bfv\cdot \tilde\Phi_{i}\dx\bigg|_{s=t_i}^{s=t_{i+1}}\\\lim_{n\rightarrow\infty}\int_{t_i}^{t_{i+1}}\int_{\mt}\Tilde\eta_{\delta_n} \tilde\bfv_{\delta_n}\cdot\partial_t\tilde\Phi_{i,\delta_n}&
+\tilde\eta_{\delta_n} \tilde\bfv_{\delta_n}\otimes \tilde\bfv_{\delta_n}:\nabla^{\tilde\bfphi_{\delta_n}}\tilde\Phi_{i,{\delta_n}}\Big)\dxt\nonumber\\
=\int_{t_i}^{t_{i+1}}\int_{\mt}\Big(\tilde\eta \tilde\bfv\cdot\partial_t\tilde\Phi_{i}
&+\tilde\eta \tilde\bfv\otimes \tilde\bfv:\nabla^{\tilde\bfphi}\tilde\Phi_{i}\Big)\dxt,
\end{align*} 
where $\tilde\Phi_i=\nabla^{\tilde\bfphi(t_i)}(-\Delta^{\tilde\bfphi})^{-1}(\tilde\eta-\int_{\mt}\tilde\eta^\Theta\dx)$.
The first relation follows directly from Proposition \ref{skorokhoddelta}. As for the second one we compute
\begin{align*}
\int_{t_i}^{t_{i+1}}&\int_{\mt}\Big(\tilde\eta_{\delta_n}\tilde\bfv_{\delta_n}\cdot\partial_t\tilde\Phi_{i,n}+\tilde\eta_{\delta_n} \tilde\bfv_{\delta_n}\otimes \tilde\bfv_{\delta_n}:\nabla^{\tilde\bfphi_{\delta_n}}\tilde\Phi_{i,n}\Big)\dxs\\
&=-\int_{t_i}^{t_{i+1}}\int_{\mt}\tilde\eta_{\delta_n} \tilde\bfv_{\delta_n}\cdot\nabla^{\tilde\bfphi_{\delta_n}(t_i)}(-\Delta^{\tilde\bfphi_{\delta_n}(t_i)})^{-1}\Div^{\tilde\bfphi_{\delta_n}}(\tilde\eta_{\delta_n}\tilde\bfv_{\delta_n})\dxs\\
&+\int_{t_i}^{t_{i+1}}\int_{\mt}\tilde\eta_{\delta_n} \tilde\bfv_{\delta_n}\otimes \tilde\bfv_{\delta_n}:\nabla^{\tilde\bfphi_{\delta_n}}\nabla^{\tilde\bfphi_{\delta_n}(t_i)}(-\Delta^{\tilde\bfphi_{\delta_n}(t_i)})^{-1}\tilde\eta_{\delta_n}\dxs\\
&=-\int_{t_i}^{t_{i+1}}\int_{\mt}\tilde\eta_{\delta_n} \tilde\bfv_{\delta_n}\cdot\nabla^{\tilde\bfphi_{\delta_n}(t_i)}(-\Delta^{\tilde\bfphi_{\delta_n}(t_i)})^{-1}\Div^{\tilde\bfphi_{\delta_n}(t_i)}(\tilde\eta_{\delta_n}\tilde\bfv_{\delta_n})\dxs\\
&+\int_{t_i}^{t_{i+1}}\int_{\mt}\tilde\eta_{\delta_n} \tilde\bfv_{\delta_n}\otimes \tilde\bfv_{\delta_n}:\nabla^{\tilde\bfphi(t_i)}\nabla^{\tilde\bfphi_{\delta_n}(t_i)}(-\Delta^{\tilde\bfphi_{\delta_n}(t_i)})^{-1}\tilde\eta_{\delta_n}\dxt\\
&+\int_{t_i}^{t_{i+1}}\int_{\mt}\tilde\eta_{\delta_n} \tilde\bfv_{\delta_n}\cdot\nabla^{\tilde\bfphi_{\delta_n}(t_i)}(-\Delta^{\tilde\bfphi_{\delta_n}(t_i)})^{-1}(\Div^{\tilde\bfphi_{\delta_n}(t_i)}-\Div^{\tilde\bfphi_{\delta_n}})(\tilde\eta_{\delta_n}\tilde\bfv_{\delta_n})\dxs\\
&+\int_{t_i}^{t_{i+1}}\int_{\mt}\tilde\eta_{\delta_n} \tilde\bfv_{\delta_n}\otimes \tilde\bfv_{\delta_n}:\big(\nabla^{\tilde\bfphi_{\delta_n}}-\nabla^{\tilde\bfphi_{\delta_n}(t_i)}\big)\nabla^{\tilde\bfphi_{\delta_n}(t_i)}(-\Delta^{\tilde\bfphi_{\delta_n}(t_i)})^{-1}\tilde\eta_{\delta_n}\dxt.
\end{align*}
On account of \eqref{eq:1012delta} in combination with Proposition \eqref{skorokhoddelta} the last two terms are arbitrarily small.
Finally, we can apply div-curl lemma Lemma~\ref{lem:divcurl} to pass to the limit in the first two terms. Comparing again
$\tilde\bfphi(t_i)$ and $\tilde\bfphi(t)$ in $(t_i,t_{i+1})$ by means of \eqref{eq:1012delta} we conclude with \eqref{eq:flux'}

The next aim is to prove that $\tilde\eta$ is a renormalized solution. Using \eqref{eq:ren:delta} with $\theta=T_k$ and passing to the limit $\delta\rightarrow0$ we arrive at
\begin{align}\label{eq:Tk''}
\partial_t \tilde T^{1,k}+\Div^{\tilde\bfphi}\big( \tilde T^{1,k}\tilde\bfv\big)+\tilde T^{2,k}= 0
\end{align}
in the sense of distributions in $I\times\mt$. 
The next step is to show $\tilde\p$-a.s.
\begin{align}\label{eq.amplosc''}
\limsup_{n\rightarrow\infty}\int_0^{T}\int_{\mt}|T_k(\tilde\eta_{\delta_n})-T_k(\tilde\eta)|^{\gamma+1}\dxt\leq C_L,
\end{align}
where $C_L$ is a deterministic constant which does not depend on $k$. The proof of \eqref{eq.amplosc''} follows exactly the arguments from the classical setting
 (see \cite{feireisl1}) using \eqref{eq:flux'} and the uniform bounds on $\tilde\bfv_{\delta_n}$, cf. Proposition \ref{skorokhoddelta}. 
In order to proceed we need a renormalised version of \eqref{eq:Tk''}. 


We have
\begin{align*}
\partial_t \theta(\tilde T^{1,k})+\Div^{\tilde\bfphi}\big( \theta(\tilde T^{1,k})\tilde\bfv\big)+\theta'(\tilde T^{1,k})\Tilde T^{2,k}=(\theta(\tilde T^{1,k})-\theta'(\tilde T^{1,k})\tilde T^{1,k})\Div^{\tilde\bfphi}\tilde\bfv.
\end{align*}
Based on \eqref{eq.amplosc''}
one can now prove that $\tilde T^{1,k}\rightarrow \tilde\eta$
and $\theta'(T^{1,k})T^{2,k}\rightarrow0$ as $k\rightarrow\infty$.
This only requires uniform bounds on $\div^{\tilde\bfphi}\tilde\bfv$ in $L^2(I\times\mt)$ and on $\tilde\eta$ in $L^\gamma(I\times\mt)$. They hold uniformly with respect to the sample space. Consequently it is not necessary to pass to expectations. Hence we have proved the renormalised continuity equation for the limit objects, i.e.,
 \begin{align*}
\partial_t \theta(\tilde\eta)+\Div^{\tilde\bfphi}\big( \theta(\tilde\eta)\tilde\bfv\big)=(\theta(\tilde\eta)-\theta'(\tilde\eta)\tilde\eta)\Div^{\tilde\bfphi}\tilde\bfv.
\end{align*}
In particular, we have
\begin{align*}
\int_{\mt} L_k(\tilde\eta)\dx\bigg|_{s=0}^{s=t}+\int_0^t\int_{\mt}T_k(\tilde\eta)\,\Div^{\tilde\bfphi}\tilde\bfv\dx\ds=0,
\end{align*}
where $L_k$ is defined in \eqref{eq:Lk}.
Similarly, passing to the limit in the renormalised
continuity equation on the approximate level with the choice $\theta(\eta)=\eta L_k(\eta)$ yields
\begin{align*}
\int_{\mt} \overline{\tilde\eta L_k(\tilde\eta)}\dx\bigg|_{s=0}^{s=t}+\int_0^t\int_{\mt}\overline{T_k(\tilde\eta)\,\Div^{\tilde\bfphi}\tilde\bfv}\dx\ds=0.
\end{align*}
Comparing both and using \eqref{eq:flux'} proves 
\begin{align*}
\int_{\mt} \big(\overline{\tilde\eta L_k(\tilde\eta)}-\tilde\eta L_k(\tilde\eta)\big)(t)\dx\leq \int_0^t\int_{\mt}(\overline{T_k(\tilde\eta)}-T_k(\tilde\eta)\big)\div^{\tilde\bfphi}\tilde\bfv\dxs,
\end{align*}
where the right-hand side converges to 0 as $k\rightarrow\infty$ by \eqref{eq.amplosc''}.
Hence we have $\overline{\tilde\eta L_k(\tilde\eta)}(t)=\big(\tilde\eta L_k(\tilde\eta)\big)(t)$
and thus $\tilde\eta_{\delta_n}(t)\rightarrow\tilde\eta(t)$ in $L^1(\mt)$ by convexity of the mapping $\tilde\eta\mapsto \tilde\eta\log(\tilde\eta)$.  This proves $\overline{p(\tilde\eta)}=p(\tilde\eta)$.
This finishes the proof of Theorem \ref{thm:main}.

\section*{Compliance with Ethical Standards}\label{conflicts}

\smallskip
\par\noindent 
{\bf Funding}. 
D.B. has been funded by Grant BR 4302/3-1 (525608987) of the German Research Foundation (DFG) within the framework of the priority research program SPP 2410 and by Grant BR 4302/5-1 (543675748) of the German Research Foundation (DFG).

The research of E.F.~leading to these results 
has received funding from
the Czech Sciences Foundation (GA\v CR), Grant Agreement
24-11034S.
The Institute of Mathematics of the Academy of Sciences of
the Czech Republic is supported by RVO:67985840.

M.H. has received funding from the European Research Council (ERC)
under the European Union’s Horizon 2020 research and innovation programme (grant agreement
No. 949981).

P.B.M. has been partially supported by the National Science Centre grant no.  2022/45/B/ST1/03432 (Opus).

\smallskip
\par\noindent
{\bf Conflict of Interest}. The authors declare that they have no conflict of interest.

\smallskip
\par\noindent
{\bf Data Availability}. Data sharing is not applicable to this article as no datasets were generated or analysed during the current study.

\end{document}